\documentclass[reqno,final]{amsart}

\usepackage{caption}
\usepackage{subcaption}
\usepackage{amsfonts,latexsym,graphicx,amsmath,amssymb,bbm,enumitem,url}
\usepackage{algorithm}
\usepackage{algpseudocode}

\makeindex

\usepackage[color,notref,notcite]{showkeys}
\allowdisplaybreaks

\definecolor{labelkey}{rgb}{0.6,0,1}

\usepackage[normalem]{ulem}
\newcounter{corr}
\definecolor{violet}{rgb}{0.580,0.,0.827}
\newcommand{\corr}[3]{\typeout{Warning : a correction remains in page
\thepage}
				\stepcounter{corr}        
				{\color{blue}\ifmmode\text{\,\sout{\ensuremath{#1}}\,}\else\sout{#1}\fi}
        {\color{red}#2}
        {\color{violet} \fbox{\thecorr}#3}}

\newcounter{cst}
\catcode`\@=11
\def\ctel#1{C_{\refstepcounter{cst}\@bsphack
\protected@write\@auxout{}%
           {\string\newlabel{#1}{{\thecst}{\thepage}}}\thecst}}
\catcode`\@=12

\newcounter{cexp}
\catcode`\@=11
\def\terml#1{T_{\refstepcounter{cexp}\@bsphack
\protected@write\@auxout{}%
           {\string\newlabel{#1}{{\thecexp}{\thepage}}}\thecexp}}
\catcode`\@=12

\newcommand{\eop}{{\unskip\nobreak\hfil\penalty50
           \hskip2em\hbox{}\nobreak\hfil\mbox{\rule{1ex}{1ex} \qquad}
   \parfillskip=0pt
   \finalhyphendemerits=0\par\medskip}}
\renewenvironment{proof}[1][]{\noindent {\bf Proof#1. } }{\eop}

\newtheorem{theorem}{Theorem}[section]
\newtheorem{remark}[theorem]{Remark}

\newtheorem{lemma}[theorem]{Lemma} 
\newtheorem{definition}[theorem]{Definition}

\definecolor{shadecolor}{gray}{0.92}
\definecolor{TFFrameColor}{gray}{0.92}
\definecolor{TFTitleColor}{rgb}{0,0,0}

\newcommand{\ba}{\begin{array}{llll}   }
\newcommand{\bac}{\begin{array}{c}}
\newcommand{\bari}{\begin{array}{r}}
\newcommand{\ea}{\end{array}}

\newcommand{\ban}{\begin{array}{llll}}
\newcommand{\ean}{\end{array}}

\newcommand{\be}{\begin{equation}}
\newcommand{\ee}{\end{equation}}

\newcommand{\beqsys }{\beqtab \left \{ \begin{array}{l}}
\newcommand{\eeqsys }{\end{array} \right . \eeqtab }

\newcommand{\benum}{\begin{enumerate}}
\newcommand{\eenum}{\end{enumerate}}

\newcommand{\beqtab}{\begin{eqnarray}} 
\newcommand{\eeqtab}{\end{eqnarray}}

\newcommand{\Vel}{\mathbf{V}}

\newcommand{\disc}{{\mathcal D}}

\newcommand{\rd}{\mathrm{d}}

\newcommand{\edges}{{\mathcal E}}              

\newcommand{\eh}{\mathrm{h}}
\newcommand{\ei}{\mathrm{i}}

\newcommand{\flux}{\mathbf{\Gamma}}

\newcommand{\hV}{V_{x,j+1/2}}

\newcommand{\mesh}{{\mathcal M}}

\newcommand{\norm}[2]{\| #1 \|_{#2}}

\newcommand{\sP}{s_{\mathrm{P}}}

\newcommand{\phy}{\varphi}

\newcommand{\R}{\mathbb R}

\newcommand{\V}{\mathbf{V}}

\renewcommand{\norm}[2]{\left\Vert#1\right\Vert_{#2}}

\usepackage{xparse}
\DeclareDocumentCommand{\RPiD}{ O{\disc} O{,0} }{\Pi_{#1}(X_{#1#2})}

\def\Fdof#1{{\bm{\mathcal{F}}(#1,\R)}}
\makeatletter
\def\Fdof{\@ifnextchar[{\@with}{\@without}}
\def\@with[#1]#2{{\bm{\mathcal{F}}(#2;#1)}}
\def\@without#1{{\bm{\mathcal{F}}(#1,\R)}}
\makeatother

\def\RT0{\mathbb{RT}_0}

\newcommand{\sCoord}{\eta}
\definecolor{labelkey}{rgb}{0.6,0,1}
\makeatletter
\newcommand*\bigcdot{\mathpalette\bigcdot@{.5}}
\newcommand*\bigcdot@[2]{\mathbin{\vcenter{\hbox{\scalebox{#2}{$\m@th#1\bullet$}}}}}

\def\BState{\State\hskip-\ALG@thistlm}
\makeatother

\begin{document}

	\title[]{Complete flux scheme for variable velocity fields: coupling between the advection-diffusion equation and the Poisson equation for the velocity field}
	
		\author{Hanz Martin Cheng}
		\address{School of Engineering Science, Lappeenranta--Lahti University of Technology, P.O. Box 20, 53851 Lappeenranta, Finland}
	\email{hanz.cheng@lut.fi} 
				\author{Jan ten Thije Boonkkamp}
		\address{Department of Mathematics and Computer Science, Eindhoven University of Technology, P.O. Box 513, 5600 MB Eindhoven, The Netherlands.}
			\email{j.h.m.tenthijeboonkkamp@tue.nl}
	
	\date{\today}
	
	%
	%
	%
	\maketitle
		\begin{abstract}
		In this work, we consider an advection-diffusion equation, coupled to a Poisson equation for the velocity field. This type of coupling is typically encountered in models arising from plasma physics or porous media flow. The aim of this work is to build upon the complete flux scheme (an improvement over the Scharfetter-Gummel scheme by considering the contribution of the source term), so that its second-order convergence, which is uniform in P\'eclet numbers, carries over to these models. This is done by considering a piecewise linear approximation of the velocity field, which is then used for defining upwind-adjusted P\'eclet numbers. 
		\end{abstract}
\section{Introduction}
We consider an advection-diffusion equation, coupled to a Poisson equation for the velocity field. This type of coupled system is encountered in models used for physical applications, such as the coupling of Darcy's law with the conservation law for the concentration of an injected solvent for the miscible flow model in porous media  \cite{CS17-incompressible-2phase-flow,E83-Mathematics-Reservoir,IHMMA16-advanced-petroleum-reservoir,P77-Reservoir-Simulation}, and the coupling of the Poisson equation to the continuity equations for electron-ion densities in fluid models for plasma physics \cite{C05-fluid_ion,H07-model,Hetal02-thruster,K03-fluid_ion,P94-model}. For these models, the advection-diffusion equations are time-dependent. In this work,  we focus on a stationary coupled model, where we drop the dependence on the time variable. The aim of this work is to develop a fully second-order numerical scheme which is uniform in P\'eclet numbers for the stationary coupled system of equations. This will serve as a first step towards the development of an efficient and accurate numerical scheme for the time-dependent coupled system.

In this paper, we consider asymptotic-preserving finite volume type schemes. That is, we require the scheme to remain consistent with the inviscid (asymptotic) limit when the diffusion parameter vanishes. A lot of work has been done in studying these types of schemes and their extension to more complex problems, see e.g. \cite{BMP89-2Dexponential,R94-Ilin,RST08-robust_num,SS99-FE_splines}, and more recently \cite{BG18-truly2D} among others. In this paper, we focus on the complete flux scheme, the main idea of which involves extending the classical Scharfetter-Gummel \cite{SG69} or Il'in \cite{Ili69-Ilin}  scheme so that the contribution of the source term is also taken into account; see e.g., \cite{FL17-FVCF-unstructured-grids,AB11-FVCF}. One of the aims of this paper is to build upon the work in \cite{FL17-FVCF-unstructured-grids}, so that we can deal with nonconstant velocity fields. This is done by considering piecewise linear approximations of the velocity field, and by introducing \emph{upwind-adjusted P\'eclet numbers}. Aside from being able to take care of nonconstant velocity fields, we demonstrate that the proposed numerical scheme is able to handle discrete velocity fields which arise from solving a Poisson equation numerically.  

The paper is structured as follows: In Section \ref{sec:Model}, we start by presenting the model, consisting of an advection-diffusion equation coupled to a Poisson equation. Following this, we discuss in Section \ref{sec:CF_1D} the complete flux scheme in one dimension, where we start by presenting the integral representation of the flux in Section \ref{sec:CF_integral_rep}. We then explore in Section \ref{sec:pwc} the complete flux scheme with a piecewise constant approximation of the velocity field. The main novelty of this paper is then presented in Section \ref{sec:pwl}, where we consider a piecewise linear approximation of the velocity field, and construct upwind-adjusted P\'eclet numbers, which are used for the modified homogeneous and inhomogeneous fluxes. Section \ref{sec:CF_coupled} then presents the application of the complete flux scheme to the coupled system of equations. In particular, sections \ref{sec:pwc_coupled}--\ref{sec:pwl_coupled} discuss the relation between the discrete fluxes obtained in sections \ref{sec:pwc}--\ref{sec:pwl} to the discrete solutions of the Poisson equation. Section \ref{sec:CF_2D} then discusses how the complete flux scheme can be applied in two dimensions. Numerical tests are then presented in Section \ref{sec:NumTests} to demonstrate the second-order accuracy of the scheme, even for extreme tests for which the Poisson equation has a very steep source term. Finally, we provide a summary in Section \ref{sec:Conc}, together with possible directions to explore for future work.

\section{Model}\label{sec:Model}
Consider an advection-diffusion equation on an open bounded domain $\Omega \subset \mathbb{R}^d$ in dimension $d$, 
\begin{subequations}
\begin{equation}\label{eq:adv-diff}
\nabla \bigcdot \big(\mu c\mathbf{V}-D\nabla c) = s, \quad x\in\Omega,
\end{equation}
coupled to a Poisson equation
\begin{equation}\label{eq:Poisson}
\nabla \bigcdot (-\nabla \phy) = \sP, \quad x\in\Omega,
\end{equation}
\end{subequations}
with unknowns $c$ and $\phy$. Here, $s, \sP$ are the source terms of the advection-diffusion equation and the Poisson equation, respectively. The parameters $D>0, \mu \in \mathbb{R}$ are constants which represent the diffusion and mobility coefficients, respectively. The velocity field is defined by the relation $\Vel=-\nabla \phy$. This type of coupling is typically encountered in physical applications, such as the coupling between the electric field and the electron-ion densities in plasma physics, or between the Darcy velocity and solvent concentration for porous media flow. We note here that for simplicity, we consider $D$ to be a scalar, however, the scheme discussed in this paper can also be applied when $D$ is an anisotropic diffusion tensor (see e.g. \cite{CT20-complete_flux}).  In this paper, we use finite volume schemes for discretising the coupled system of equations. In particular, we consider the complete flux scheme and how it can be effectively applied onto coupled systems of equations. To start off, we consider the complete flux scheme in one dimension.
\section{Complete flux scheme in one dimension}\label{sec:CF_1D}
\subsection{Integral representation of the flux} \label{sec:CF_integral_rep}
Before presenting the numerical scheme, we give an integral representation of the flux. We start by writing the system \eqref{eq:adv-diff}--\eqref{eq:Poisson} in one dimension, 
\begin{subequations}
\begin{equation}\label{eq:adv-diff_1D}
\frac{\rd}{\rd x}\bigg(\mu c V-D\frac{\rd c}{\rd x}\bigg) = s, \quad  x\in\Omega,
\end{equation}
\begin{equation}\label{eq:Poisson_1D}
-\frac{\rd^2\phy}{\rd x^2} = \sP, \quad x\in\Omega.
\end{equation}
\end{subequations}
Considering $\Omega = (0,L)$, we now form a partition $0=x_1<x_2<\dots<x_{N+1}=L$ of $\Omega$. We then define control volumes $K_j:=(x_{j-1/2},x_{j+1/2}), j=2,\dots, N$, where $x_{j+1/2}=(x_j+x_{j+1})/2$ is the midpoint of the interval $(x_{j},x_{j+1})$. Key to the definition of finite volume schemes is the computation of the flux
\begin{equation}\label{def:flux}
f= \mu cV - D\frac{\rd c}{\rd x}
\end{equation} 
at the interfaces $x_{j-1/2}$ and $x_{j+1/2}$ of $K_j$. For this work, we use the complete flux scheme. We start by describing the complete flux scheme for \eqref{eq:adv-diff_1D} with a given velocity field $V$, and relate it to a velocity reconstructed from the discrete solution of \eqref{eq:Poisson_1D} afterwards. Integrating \eqref{eq:adv-diff_1D} over each control volume $K_j$, the finite volume scheme then requires us to find $c$ such that for $j=2,\dots, N$,
\begin{equation}\label{eq:FV_scheme}
f_{j+1/2}-f_{j-1/2} = \int_{x_{j-1/2}}^{x_{j+1/2}} s \,\mathrm{d}x,
\end{equation} 
where $f_{j+1/2}$ and $f_{j-1/2}$ are the values of the flux $f$ at $x_{j+1/2}$ and $x_{j-1/2}$, respectively. For the complete flux scheme, the fluxes are computed by solving, for each sub-interval $(x_j,x_{j+1})$, the boundary value problem
\begin{equation}\nonumber
\begin{aligned}
 &\frac{\rd f}{\rd x} = \frac{\rd}{\rd x}\bigg(\mu c V-D\frac{\rd c}{\rd x}\bigg) = s \quad \mathrm{on} \quad (x_j,x_{j+1}),\\
&c(x_j) = c_j, \quad c(x_{j+1})=c_{j+1}.
\end{aligned}
\end{equation}
Denoting $\Delta x = x_{j+1}-x_j$, we introduce the scaled coordinate 
\[\sCoord= \frac{x-x_j}{\Delta x}.
\] 
Here, $\sCoord\in[0,1]$, and we may rewrite the boundary value problem as
\begin{subequations}

\begin{align}
&f'=\bigg(\mu c V-\frac{D}{\Delta x}c'\bigg)' = s\Delta x \quad \mathrm{on} \quad (0,1), \label{eq:BVPscaled}\\
&c(0) = c_j, \quad c(1)=c_{j+1}, \label{eq:BC_BVP}
\end{align}
\end{subequations}
where $'$ denotes differentiation with respect to $\sCoord$. We note here that even though we choose a uniform mesh with $\Delta x=x_{j+1}-x_j$ for $j=1,\dots,N$, the method presented below can easily be adapted to non-uniform meshes with $\Delta x_j:= x_{j+1}-x_j$, $j=1,\dots,N$. Defining the local P\'eclet function 
\[
\Lambda(\eta) = \Delta x \int_{1/2}^{\eta} \frac{\mu}{D}V  \,\mathrm{d}\zeta,
\]
 we can also rewrite the flux \eqref{def:flux} for the advection-diffusion equation in the scaled coordinate
\begin{equation}\nonumber
\begin{aligned}
f(\sCoord)&=\mu c V-\frac{D}{\Delta x}c'\\
&= -\frac{D}{\Delta x}\bigg(c'-\frac{\mu}{D} cV \Delta x  \bigg)\\
&= -\frac{D}{\Delta x}\big(c\,e^{-\Lambda(\sCoord)}\big)'e^{\Lambda(\sCoord)}.
\end{aligned}
\end{equation} 
Taking then the integral of $f'=s\Delta x$ from $1/2$ to $\eta\in(0,1)$, using \eqref{eq:BVPscaled} and denoting $f_{j+1/2} = f(1/2)$ gives us
\begin{equation}\nonumber
\begin{aligned}
f(\sCoord)-f_{j+1/2} &= \hat{s}(\sCoord),\\
\hat{s}(\sCoord) &= \Delta x \int_{1/2}^\sCoord s \,\mathrm{d}\zeta,
\end{aligned} 
\end{equation}
or equivalently,
\begin{equation} \label{eq:flux_exact_diff}
\big(ce^{-\Lambda(\sCoord)}\big)' = -\frac{e^{-\Lambda(\sCoord)}}{D} \Delta x (f_{j+1/2}+\hat{s}(\sCoord)).
\end{equation}
For finite volume methods, it is common to approximate the source term with piecewise constants, i.e., for $j=2,\dots,N$,
\begin{equation}\nonumber
s = s_j, \quad x\in K_j,
\end{equation}
which gives us 
\begin{equation} \nonumber
\hat{s}(\sCoord)=\begin{cases}
s_j\Delta x(\sCoord-\frac12) &\quad \sCoord\in(0,1/2), \\
s_{j+1}\Delta x(\sCoord-\frac{1}{2}) & \quad \sCoord \in(1/2,1).
\end{cases}
\end{equation}
Integrating \eqref{eq:flux_exact_diff} from $0$ to $1$, applying the boundary conditions \eqref{eq:BC_BVP}, and solving for $f_{j+1/2}$ then leads to 
\begin{subequations}
\begin{align}
f_{j+1/2}&= f_{j+1/2}^\eh + f_{j+1/2}^\ei, \label{eq:flux_exact}
\end{align}
where
\begin{align}
f_{j+1/2}^\eh &= -\frac{D}{\Delta x}\frac{c_{j+1}e^{-\Lambda(1)}-c_{j}e^{-\Lambda(0)}}{\int_{0}^{1}e^{-\Lambda(\sCoord)} \mathrm{d}\sCoord},\label{eq:HF_exact}\\
f_{j+1/2}^\ei &= -\Delta x \frac{s_j \int_{0}^{1/2} (\sCoord-1/2)e^{-\Lambda(\sCoord)} \mathrm{d}\sCoord+s_{j+1}  \int_{1/2}^{1} (\sCoord-1/2)e^{-\Lambda(\sCoord)} \mathrm{d}\sCoord}{\int_{0}^{1}e^{-\Lambda(\sCoord)} \mathrm{d}\sCoord}, \label{eq:IF_exact}
\end{align}
\end{subequations}
are the homogeneous and inhomogeneous components of the flux, respectively.

We now note that there are several integrals that need to be computed. Firstly, we need to determine the values $\Lambda(0)$ and $\Lambda(1)$. Secondly, the integral $\int_{0}^{1}e^{-\Lambda(\sCoord)} \mathrm{d}\sCoord$, and finally the integrals $\int_{0}^{1/2}(\sCoord-1/2)e^{-\Lambda(\sCoord)} \mathrm{d}\sCoord$ and $\int_{1/2}^{1}(\sCoord-1/2)e^{-\Lambda(\sCoord)} \mathrm{d}\sCoord$. Since $V$ is nonconstant, these values cannot be computed exactly, and approximations have to be made. We discuss different choices for the approximations of these quantities. 
\subsection{Piecewise constant approximation of the velocity}\label{sec:pwc}
In this section, we consider a piecewise constant approximation of $V$. That is, on each interval $(x_j,x_{j+1})$, we approximate $V$ by its value $V_{j+1/2}$ at the midpoint $x_{j+1/2}$. In this case, the homogeneous and inhomogeneous fluxes \eqref{eq:HF_exact}--\eqref{eq:IF_exact} can be evaluated exactly, and we find that 
\begin{subequations}\label{eq:fluxes_std}
	\begin{align}
	f_{j+1/2}^\eh &= -\frac{D}{\Delta x}\big(B(\mathrm{Pe}_{j+1/2})c_{j+1}-B(-\mathrm{Pe}_{j+1/2})c_j\big),\label{eq:HF_std}\\
	f_{j+1/2}^\ei &= -\Delta x( W(\mathrm{Pe}_{j+1/2})s_{j+1} -W(-\mathrm{Pe}_{j+1/2})s_j ), \label{eq:IF_std}
	\end{align}
\end{subequations}
where
\begin{equation}\label{eq:Pec_std}
\mathrm{Pe}_{j+1/2} = \frac{\mu}{D}V_{j+1/2} \Delta x 
\end{equation}
is the P\'eclet number at the interface $x_{j+1/2}$, 
\begin{equation}\label{eq:Bern}
B(z) = \frac{z}{e^z-1}
\end{equation}
is the Bernoulli function, and 
\begin{equation}\label{eq:W}
W(z) = \frac{e^{z/2}-1-\frac{1}{2}z}{z(e^z-1)}.
\end{equation}

Using the approximations \eqref{eq:HF_std}--\eqref{eq:IF_std} for the fluxes in the finite volume scheme, we obtain a uniformly second-order accurate method for constant velocity fields $V$, see e.g. \cite{FL17-FVCF-unstructured-grids,AB11-FVCF}. However, the discrete fluxes \eqref{eq:HF_std}--\eqref{eq:IF_std} are no longer accurate enough when $V$ is nonconstant. One way to maintain second-order accuracy would be to use quadrature rules for evaluating the integrals in \eqref{eq:HF_exact}--\eqref{eq:IF_exact}, as suggested in \cite{AB11-FVCF}. In this paper, we consider a piecewise linear approximation of the velocity.

\subsection{Piecewise linear approximation of the velocity}\label{sec:pwl}

On each interval $(x_j,x_{j+1})$, we write a Taylor expansion of the velocity $V$ centred at $x_{j+1/2}$. Denoting by $\hV = \frac{\rd V}{\rd x}(x_{j+1/2})$, we write
\begin{equation}\nonumber
\begin{aligned}
V(x) &= V_{j+1/2} + \hV(x-x_{j+1/2}) + \mathcal{O}(\Delta x^2) \\
&= V_{j+1/2} + \alpha_j \hV(x-x_{j+1/2}) + (1-\alpha_j) \hV(x-x_{j+1/2}) + \mathcal{O}(\Delta x^2).
\end{aligned}
\end{equation}
Here, we introduce the parameter $\alpha_j\in[0,1]$ and split the linear term involving $\hV$ into two components with factor $\alpha_j$ and $(1-\alpha_j)$, respectively. In terms of the scaled coordinate $\sCoord$, when we drop the term with factor $(1-\alpha_j)$, we obtain the approximation
\begin{equation}\label{eq:V_pcwise_lin}
V(\sCoord) = V_{j+1/2} + \alpha_j \hV\Delta x(\sCoord-1/2), \quad \sCoord\in(0,1).
\end{equation}
From here onwards $V(\sCoord)$ will refer to the piecewise linear approximation of the velocity \eqref{eq:V_pcwise_lin}. We note here that the presence of the parameter $\alpha_j\in[0,1]$ allows us to transition between first and second-order approximations for $V$. The main purpose of $\alpha_j$ is to serve as a slope limiter, so that the term involving $\hV$ does not change the sign of $V(\sCoord)$ for $\sCoord\in(0,1)$. This will be discussed in more detail shortly in this section.  Defining 
\begin{equation}\label{eq:Q}
Q_{j+1/2} =   \frac{\mu}{D} \hV \frac{\Delta x^2}{2},
\end{equation}
we can then write
\begin{align}
\Lambda(\sCoord) &= \mathrm{Pe}_{j+1/2}(\sCoord-1/2)+ \alpha_j Q_{j+1/2} (\sCoord-1/2)^2. \nonumber 
\end{align}
Here, $\Lambda(\sCoord)$ is written as a polynomial centred at $\sCoord=1/2$. In order to consider P\'eclet numbers in the upwind direction, we evaluate $V(\sCoord)$ at the upwind boundary of the interval $(0,1)$, which is either located at $\sCoord=0$ or $\sCoord=1$, depending on the sign of $V_{j+1/2}$. To this end, we define \emph{upwind-adjusted P\'eclet numbers}. 
\begin{subequations}\label{eq:Peclet_mod}
\begin{definition} The upwind-adjusted P\'eclet numbers are defined as
	\begin{align}
 \mathrm{Pe}_{j+1/2}^+  &:= \frac{\mu}{D}V(0)\Delta x =  \mathrm{Pe}_{j+1/2}- \alpha_j Q_{j+1/2}, \qquad V_{j+1/2}>0, \label{eq:Peclet_pos}\\
	\mathrm{Pe}_{j+1/2}^- &:= \frac{\mu}{D}V(1)\Delta x = \mathrm{Pe}_{j+1/2}+ \alpha_j Q_{j+1/2}, \qquad V_{j+1/2}<0. \label{eq:Peclet_neg}
	\end{align}
\end{definition}
\end{subequations}
One interesting observation is that the modified P\'eclet numbers \eqref{eq:Peclet_mod} now contain the derivative $\hV$. We also note that $\mathrm{Pe}_{j+1/2}^+$ uses the approximate velocity \eqref{eq:V_pcwise_lin} evaluated at $\sCoord=0$, which is fully second-order when $\alpha_j=1$. Similarly, $\mathrm{Pe}_{j+1/2}^-$ uses the approximate velocity \eqref{eq:V_pcwise_lin} evaluated at $\sCoord=1$. Moreover, we note that the choice of using $\mathrm{Pe}_{j+1/2}^+$ for $V_{j+1/2}>0$ and $\mathrm{Pe}_{j+1/2}^-$ for $V_{j+1/2}<0$ is physically appropriate, as this means that we take the proper value of the P\'eclet number, in the upwind direction. We note, however, that we do not want to perform an over-correction. That is, if $V_{j+1/2}>0$, we require $\mathrm{Pe}_{j+1/2}^+>0$, and similarly, if $V_{j+1/2}<0$, we must have $\mathrm{Pe}_{j+1/2}^-<0$. 

\begin{lemma}[Conditions on the slope limiter and mesh size]
Let the P\'eclet number $\mathrm{Pe}_{j+1/2}$ at an interface $x_{j+1/2}$ be defined as in \eqref{eq:Pec_std}. Then, the upwind-adjusted P\'eclet number \eqref{eq:Peclet_mod} has the same sign as  $\mathrm{Pe}_{j+1/2}$ if and only if 
\begin{equation}\label{eq:parameter_restriction}
\alpha_j \leq \left\lvert\frac{\mathrm{Pe}_{j+1/2}}{Q_{j+1/2}}\right\rvert.
\end{equation} 
As a consequence, when we use a fully second-order approximation of the velocity, corresponding to taking $\alpha_j=1$, we require the following restriction on the mesh size
\begin{equation}\label{eq:mesh_size}
\Delta x \leq 2 \left\lvert\frac{V_{j+1/2}}{\hV}\right\rvert.
\end{equation} 
\end{lemma}
\begin{proof}
    Consider the case $\mathrm{Pe}_{j+1/2} \geq 0$. Then, $\mathrm{Pe}_{j+1/2}^+ \geq 0 $ if and only if
    \begin{equation}\nonumber
    \begin{aligned}
        \mathrm{Pe}_{j+1/2} & \geq \alpha_j Q_{j+1/2}. \\
    \end{aligned}
    \end{equation}
    If $Q_{j+1/2} <0$, then the inequality holds for any $\alpha_j\in [0,1]$. On the other hand, for $Q_{j+1/2} >0$, we may divide both sides of the inequality by $Q_{j+1/2}$ to obtain an upper bound for $\alpha_j$. Using a similar argument for the case  $\mathrm{Pe}_{j+1/2} < 0$ and combining the results then gives \eqref{eq:parameter_restriction}. The restriction \eqref{eq:mesh_size} on the mesh size follows by substituting the expressions \eqref{eq:Pec_std} and \eqref{eq:Q} into the inequality \eqref{eq:parameter_restriction} with $\alpha_j=1$. 
\end{proof}
In some instances, especially when the velocity is very steep throughout the domain, \eqref{eq:mesh_size} is too restrictive in terms of practicability as it would require a very fine mesh. Hence, for the numerical tests in Section \ref{sec:NumTests}, we simply choose 
\begin{equation}\label{eq:alpha_val}
\alpha_j = \min\bigg(1,\left\lvert\frac{\mathrm{Pe}_{j+1/2}}{Q_{j+1/2}}\right\rvert \bigg).
\end{equation}
In order to use the upwind-adjusted P\'eclet numbers, it is useful to write $\Lambda(\sCoord)$ as a polynomial centred at $\sCoord=0$ or $\sCoord=1$, given by
\begin{align}
\Lambda(\sCoord)
&=  -\frac{1}{2}\mathrm{Pe}_{j+1/2}^+ - \frac{1}{4}\alpha_j Q_{j+1/2} + \mathrm{Pe}_{j+1/2}^+\sCoord+ \alpha_j Q_{j+1/2}\sCoord^2 \nonumber \\
&=  \frac{1}{2}\mathrm{Pe}_{j+1/2}^- - \frac{1}{4}\alpha_jQ_{j+1/2}  + 
\mathrm{Pe}_{j+1/2}^-(\sCoord-1)+\alpha_jQ_{j+1/2} (\sCoord-1)^2. \nonumber
\end{align}
\subsubsection{Homogeneous flux} 
We now compute the homogeneous flux $f_{j+1/2}^\eh$ in \eqref{eq:HF_exact}. We start by
rewriting the homogeneous flux in the following form 
\begin{equation}\label{eq:HF_modP} 
f_{j+1/2}^\eh = -\frac{D}{\Delta x} \frac{c_{j+1}e^{-\mathrm{Pe}_{j+1/2}}-c_j}{\int_{0}^1 e^{\Lambda(0)-\Lambda(\sCoord)}\,\mathrm{d}\sCoord}. 
\end{equation}
The main challenge then comes with evaluating the integral $\int_{0}^1 e^{\Lambda(0)-\Lambda(\sCoord)}\,\mathrm{d}\sCoord$. One property we want to maintain is that if $\hV=0$, corresponding to $Q_{j+1/2}=0$, then the homogeneous flux should reduce to \eqref{eq:HF_std}. To this end, we perform an integration by parts and write three equivalent alternatives for the integral, given by
\begin{subequations}
\begin{align}
\int_{0}^1 e^{\Lambda(0)-\Lambda(\sCoord)}\,\mathrm{d}\sCoord&= \frac{1}{\mathrm{Pe}_{j+1/2}}\big(1- e^{-\mathrm{Pe}_{j+1/2}}\big) \nonumber\\
&\,\,\,- \frac{2\alpha_jQ_{j+1/2}}{ \mathrm{Pe}_{j+1/2}}  e^{\Lambda(0)}\int_{0}^1 (\sCoord-1/2) e^{-\mathrm{Pe}_{j+1/2}(\sCoord-1/2)} e^{-\alpha_jQ_{j+1/2}(\sCoord-1/2)^2}\,\mathrm{d}\sCoord \nonumber\\
 &= \frac{1}{\mathrm{Pe}_{j+1/2}^+}\big(1- e^{-\mathrm{Pe}_{j+1/2}}\big) \nonumber\\
&\,\,\,- \frac{2\alpha_jQ_{j+1/2}}{\mathrm{Pe}_{j+1/2}^+ }  \int_{0}^1 \sCoord e^{-\mathrm{Pe}_{j+1/2}^+\sCoord} e^{-\alpha_jQ_{j+1/2}\sCoord^2}\,\mathrm{d}\sCoord \nonumber\\
&= \frac{1}{\mathrm{Pe}_{j+1/2}^-}\big(1- e^{-\mathrm{Pe}_{j+1/2}}\big) \nonumber\\
&\,\,\,- \frac{2\alpha_jQ_{j+1/2}}{ \mathrm{Pe}_{j+1/2}^-}  e^{-\mathrm{Pe}_{j+1/2}}\int_{0}^1 (\sCoord-1) e^{-\mathrm{Pe}_{j+1/2}^-(\sCoord-1)} e^{-\alpha_jQ_{j+1/2}(\sCoord-1)^2}\,\mathrm{d}\sCoord\nonumber.
\end{align}
\end{subequations}
Here, we notice three expressions for $\int_{0}^1 e^{\Lambda(0)-\Lambda(\sCoord)}\,\mathrm{d}\sCoord $, each of which is obtained by integration by parts whilst considering $\Lambda(\sCoord)$ as a polynomial centred at $\sCoord = 1/2 , 0, 1,$ respectively. One important property to note is that regardless of the quadrature rule used for evaluating the integrals for all three cases, we recover the homogeneous flux \eqref{eq:HF_std} whenever $\hV=0$. Using now the trapezoidal rule for approximating the integrals, we obtain three discrete approximations for  $\int_{0}^1 e^{\Lambda(0)-\Lambda(\sCoord)}\,\mathrm{d}\sCoord $, given by
\begin{subequations}
\begin{align}
\int_{0}^1 e^{\Lambda(0)-\Lambda(\sCoord)}\,\mathrm{d}\sCoord &\approx \frac{1- e^{-\mathrm{Pe}_{j+1/2}}}{\mathrm{Pe}_{j+1/2}}\bigg(1+\frac{\alpha_j}{2}Q_{j+1/2}\bigg),\nonumber \\
\int_{0}^1 e^{\Lambda(0)-\Lambda(\sCoord)}\,\mathrm{d}\sCoord &\approx \frac{1- e^{-\mathrm{Pe}_{j+1/2}}}{\mathrm{Pe}_{j+1/2}^+}
- \frac{\alpha_jQ_{j+1/2}e^{-\mathrm{Pe}_{j+1/2}}}{ \mathrm{Pe}_{j+1/2}^+ }, \nonumber \\
 \int_{0}^1 e^{\Lambda(0)-\Lambda(\sCoord)}\,\mathrm{d}\sCoord &\approx \frac{1- e^{-\mathrm{Pe}_{j+1/2}}}{\mathrm{Pe}_{j+1/2}^-} +\frac{\alpha_jQ_{j+1/2}}{ \mathrm{Pe}_{j+1/2}^-},\nonumber
\end{align}
\end{subequations}
which, upon substitution into \eqref{eq:HF_modP}, result in the following discretisations of the homogeneous flux
\begin{subequations}
	\begin{align}
	f_{j+1/2}^\eh &=- \frac{D}{\Delta x} \frac{1}{1+\frac{\alpha_j}{2} Q_{j+1/2}} (B(\mathrm{Pe}_{j+1/2}) c_{j+1} - B(-\mathrm{Pe}_{j+1/2}) c_{j}),\label{eq:HF_mod_IBP_ctr} \\
f_{j+1/2}^{\eh,+} &=- \frac{D}{\Delta x} \bigg(\frac{e^{-\mathrm{Pe}_{j+1/2}}\mathrm{Pe}_{j+1/2}^+}{1-(1+\alpha_jQ_{j+1/2})e^{-\mathrm{Pe}_{j+1/2}}} c_{j+1} - \frac{\mathrm{Pe}_{j+1/2}^+}{1-(1+\alpha_jQ_{j+1/2})e^{-\mathrm{Pe}_{j+1/2}}}c_j\bigg), \label{eq:HF_mod_IBP_pos}\\
f_{j+1/2}^{\eh,-} &=	- \frac{D}{\Delta x} \bigg(\frac{\mathrm{Pe}_{j+1/2}^-}{(1+\alpha_jQ_{j+1/2})e^{\mathrm{Pe}_{j+1/2}}-1} c_{j+1} - \frac{e^{\mathrm{Pe}_{j+1/2}}\mathrm{Pe}_{j+1/2}^-}{(1+\alpha_jQ_{j+1/2})e^{\mathrm{Pe}_{j+1/2}}-1}c_j\bigg).\label{eq:HF_mod_IBP_neg}
	\end{align}
\end{subequations}
As a remark, we see here that  as opposed to \eqref{eq:HF_std}, the discrete homogeneous fluxes \eqref{eq:HF_mod_IBP_ctr}--\eqref{eq:HF_mod_IBP_neg}  take into account the effect of $\hV$. At this stage, we note that the expression \eqref{eq:HF_mod_IBP_ctr} scales the homogeneous flux \eqref{eq:HF_std} by a factor $(1+\frac{\alpha_j}{2} Q_{j+1/2})^{-1}$. However, since we do not have any control over $Q_{j+1/2}$, this scaling factor in \eqref{eq:HF_mod_IBP_ctr} might lead to an unstable numerical scheme. In particular, one problematic case is encountered if $|1+\frac{\alpha_j}{2} Q_{j+1/2}|\ll 1$, which causes the denominator of \eqref{eq:HF_mod_IBP_ctr} to vanish, leading to the coefficients of $c_j$ and $c_{j+1}$ blowing up. Another problem occurs if the velocity $V$ is very steep, corresponding to $|Q_{j+1/2}|\gg 1$. In this case, either both coefficients of $c_j$ and $c_{j+1}$ vanish, or the sign of the homogeneous flux \eqref{eq:HF_mod_IBP_ctr} is opposite the sign of \eqref{eq:HF_std}. These lead to an unstable numerical scheme unless $\alpha_j=0$. For this reason, we will no longer consider the discrete fluxes \eqref{eq:HF_mod_IBP_ctr} obtained by writing $\Lambda$ as a polynomial centred at $\sCoord=1/2$. On the other hand, the fluxes $ f_{j+1/2}^{\eh,+} $ and $f_{j+1/2}^{\eh,-}$ use the P\'eclet numbers $\mathrm{Pe}_{j+1/2}^+$ and $\mathrm{Pe}_{j+1/2}^-$, respectively, which implies that $ f_{j+1/2}^{\eh,+} $ is the proper choice when $V_{j+1/2}>0$, and $ f_{j+1/2}^{\eh,-} $ otherwise. We now look at the asymptotic behaviour of  the discrete homogeneous  fluxes as $\Delta x \rightarrow 0$. 

\begin{lemma}[Asymptotic behaviour of the discrete homogeneous fluxes] \label{lem:fluxes_IBP_HF_asymp} The discrete homogeneous fluxes \eqref{eq:HF_mod_IBP_pos} and \eqref{eq:HF_mod_IBP_neg} may be computed with the expressions 
\begin{subequations} \label{eq:HF_IBP_asymp}
 \begin{align}
f_{j+1/2}^{\eh,+} &= -\frac{D}{\Delta x} \big(e^{-\alpha_jQ_{j+1/2}} B(\mathrm{Pe}_{j+1/2}^+)c_{j+1} -B(-\mathrm{Pe}_{j+1/2}^+)c_j \big), \label{eq:HF_IBP_pos_asymp}\\
	f_{j+1/2}^{\eh,-}  &= -\frac{D}{\Delta x} \big( B(\mathrm{Pe}_{j+1/2}^-)c_{j+1} -e^{-\alpha_jQ_{j+1/2}}B(-\mathrm{Pe}_{j+1/2}^-)c_j \big), \label{eq:HF_IBP_neg_asymp}
\end{align}
respectively, in the asymptotic regime as $\Delta x \rightarrow 0$. 
\end{subequations}
\end{lemma}
\begin{proof}
 We start by considering $f_{j+1/2}^{\eh,+}$ and compute the ratio between the coefficients of $c_j$ in \eqref{eq:HF_mod_IBP_pos} and \eqref{eq:HF_IBP_pos_asymp} and take the limit as $\Delta x \rightarrow 0$. That is, we compute
 \begin{equation}\nonumber
 \begin{aligned}
     \lim_{\Delta x \rightarrow 0} \frac{(1+\alpha_j Q_{j+1/2})e^{- \mathrm{Pe}_{j+1/2}}-1}{e^{\alpha_jQ_{j+1/2}}e^{- \mathrm{Pe}_{j+1/2}}-1}.
 \end{aligned}
 \end{equation}
 Writing \begin{equation}\nonumber
e^{\alpha_j Q_{j+1/2}} = 1+\alpha_jQ_{j+1/2} + \mathcal{O}(\Delta x^4),
\end{equation}
we then have,  upon performing a Taylor expansion on $e^{- \mathrm{Pe}_{j+1/2}}$ and cancelling out the coefficients of order $\Delta x$,
 \begin{equation}\nonumber
 \begin{aligned}
   & \lim_{\Delta x \rightarrow 0} \frac{(1+\alpha_j Q_{j+1/2})\left(1-\mathrm{Pe}_{j+1/2}+\mathcal{O}(\Delta x^2)\right) -1}{\left(1+\alpha_j Q_{j+1/2} + \mathcal{O}(\Delta x^4)\right)\left(1-\mathrm{Pe}_{j+1/2}+\mathcal{O}(\Delta x^2)\right) -1}\\
     &\!\!\!=\lim_{\Delta x \rightarrow 0} \frac{-\mu V_{j+1/2}/ D+\mathcal{O}(\Delta x)}{-\mu V_{j+1/2}/ D+\mathcal{O}(\Delta x)} \\
     &\!\!\!= 1.
 \end{aligned}
 \end{equation}
 Similarly, as $\Delta x \rightarrow  0$, the ratio between the coefficients of $c_{j+1}$ in \eqref{eq:HF_mod_IBP_pos} and \eqref{eq:HF_IBP_pos_asymp} approaches 1. Thus, the discrete fluxes \eqref{eq:HF_mod_IBP_pos} may be computed by  \eqref{eq:HF_IBP_pos_asymp}. A similar proof can be done for the expressions involving $f_{j+1/2}^{\eh,-}$. 
\end{proof}

We now perform a detailed comparison of the homogeneous fluxes \eqref{eq:HF_IBP_pos_asymp} and \eqref{eq:HF_std} for the case $V_{j+1/2}>0$. We start by taking note that the ratio between the coefficients of $c_j$ and $c_{j+1}$ for both \eqref{eq:HF_IBP_pos_asymp} and \eqref{eq:HF_std} are the same. That is,
\[
\frac{B(-\mathrm{Pe}_{j+1/2})}{B(\mathrm{Pe}_{j+1/2})} = \frac{B(-\mathrm{Pe}_{j+1/2}^+)}{e^{-\alpha_jQ_{j+1/2}}B(\mathrm{Pe}_{j+1/2}^+)} = e^{\mathrm{Pe}_{j+1/2}}.
\]
This means that both homogeneous fluxes \eqref{eq:HF_IBP_pos_asymp} and \eqref{eq:HF_std} give more importance to the upwind value $c_j$ as compared to the downwind value $c_{j+1}$ by a factor of $e^{\mathrm{Pe}_{j+1/2}}$. Now, considering the case $Q_{j+1/2}>0$, we have $\mathrm{Pe}_{j+1/2}^+<\mathrm{Pe}_{j+1/2}$. Upon comparing the coefficients of $c_j$ and noting that the Bernoulli function \eqref{eq:Bern} is decreasing, we see that $B(-\mathrm{Pe}_{j+1/2})>B(-\mathrm{Pe}_{j+1/2}^+)$. This shows us that the coefficient of $c_j$, and consequently, that of $c_{j+1}$ in \eqref{eq:HF_IBP_pos_asymp} are smaller than their counterparts in \eqref{eq:HF_std}. The decrease in the coefficient of $c_j$ can be explained by the fact that $Q_{j+1/2}>0$ implies that $V_{j+1/2}$ is an overestimation of the actual velocity $V(0)$, and hence decreasing the coefficient of $c_j$ takes into account the correct weight in the upwind direction. On the other hand, for the case $Q_{j+1/2}<0$, it can be shown that the coefficients of $c_j$ and $c_{j+1}$ in \eqref{eq:HF_IBP_pos_asymp} are larger than their corresponding counterparts in \eqref{eq:HF_std}. Since $Q_{j+1/2}<0$, the increase in the coefficient of $c_{j+1}$ takes care of the fact that $V_{j+1/2}$ underestimates the actual velocity $V(0)$. A similar argument can be used to illustrate that the fluxes \eqref{eq:HF_IBP_neg_asymp} are an improvement over \eqref{eq:HF_std} for the case $V_{j+1/2}<0$.  The advantage of these modified fluxes will be illustrated in the tests in Section \ref{sec:NumTests}. 
\subsubsection{Inhomogeneous flux}
We now consider the inhomogeneous flux \eqref{eq:IF_exact}. As with the computation of the homogeneous flux, we require that $f_{j+1/2}^\ei$ reduces to \eqref{eq:IF_std} whenever $\hV=0$. To this end, we start once more with an integration by parts, followed by the trapezoidal rule. Upon expressing $\Lambda$ as a polynomial centred at $\sCoord=0,1$, respectively, and denoting 

\[P_{j+1/2}=\frac{1}{2}\mathrm{Pe}_{j+1/2}-\frac{1}{4}\alpha_j Q_{j+1/2},
\] 
we obtain the discrete inhomogeneous fluxes 
\begin{subequations}

\begin{align}
f_{j+1/2}^{\ei,+} &= -\Delta x \bigg(\frac{1-\frac{1}{2}\mathrm{Pe}_{j+1/2}^+-(1+\alpha_j\frac{Q_{j+1/2}}{2})e^{-P_{j+1/2}}}{\mathrm{Pe}_{j+1/2}^+\big((1+\alpha_jQ_{j+1/2})e^{-\mathrm{Pe}_{j+1/2}}-1\big)} s_j \nonumber\\
&\,\,\, - \frac{\big(1+\frac{\mathrm{Pe}_{j+1/2}^+}{2}\big)(1+\alpha_jQ_{j+1/2})e^{-\mathrm{Pe}_{j+1/2}}- (1-\alpha_j\frac{Q_{j+1/2}}{2})e^{-P_{j+1/2}}}{\mathrm{Pe}_{j+1/2}^+\big((1+\alpha_jQ_{j+1/2})e^{-\mathrm{Pe}_{j+1/2}}-1\big)}s_{j+1} \bigg),\label{eq:IF_mod_IBPpos}\\
f_{j+1/2}^{\ei,-}&= -\Delta x \bigg(\frac{\big(1-\frac{1}{2}\mathrm{Pe}_{j+1/2}^-\big)(1+\alpha_jQ_{j+1/2})e^{\mathrm{Pe}_{j+1/2}}-(1-\alpha_j\frac{Q_{j+1/2}}{2})e^{P_{j+1/2}}}{\mathrm{Pe}_{j+1/2}^-\big((1+\alpha_jQ_{j+1/2})e^{\mathrm{Pe}_{j+1/2}}-1\big)} s_j \nonumber\\
&\,\,\, + \frac{(1+\alpha_j\frac{Q_{j+1/2}}{2})e^{P_{j+1/2}}-\big(1+\frac{\mathrm{Pe}_{j+1/2}^-}{2}\big)}{\mathrm{Pe}_{j+1/2}^-\big((1+\alpha_jQ_{j+1/2})e^{\mathrm{Pe}_{j+1/2}}-1\big)}s_{j+1} \bigg)\label{eq:IF_mod_IBPneg}.
\end{align}
\end{subequations}
We also notice that the inhomogeneous fluxes \eqref{eq:IF_mod_IBPpos}--\eqref{eq:IF_mod_IBPneg} are natural improvements of \eqref{eq:IF_std} in the sense that they take into account the effect of $\hV$. Moreover, if $\hV=0$,  these discrete inhomogeneous fluxes reduce to \eqref{eq:IF_std}.

\begin{lemma}[Asymptotic behaviour of the discrete inhomogeneous fluxes]
In the asymptotic regime, as $\Delta x \rightarrow 0$, the discrete inhomogeneous fluxes \eqref{eq:IF_mod_IBPpos} and \eqref{eq:IF_mod_IBPneg} may be computed with the expressions 
\begin{subequations} \label{eq:IF_IBP_asymp}
 \begin{align}
f_{j+1/2}^{\ei,+} &= -\Delta x \bigg(\widetilde{W}(\mathrm{Pe}_{j+1/2}^+,-\frac{3}{4}\alpha_jQ_{j+1/2}) s_{j+1}  -\widetilde{W}(-\mathrm{Pe}_{j+1/2}^+,\frac{1}{4}\alpha_jQ_{j+1/2}) s_j \bigg),\label{eq:IF_IBP_pos_asymp}\\ 
	f_{j+1/2}^{\ei,-} &= -\Delta x \bigg(  \widetilde{W}(\mathrm{Pe}_{j+1/2}^-,-\frac{1}{4}\alpha_jQ_{j+1/2}) s_{j+1}-\widetilde{W}(-\mathrm{Pe}_{j+1/2}^-,-\frac{5}{4}\alpha_jQ_{j+1/2}) s_j \bigg),\label{eq:IF_IBP_neg_asymp}
\end{align}
\end{subequations}
respectively, where
\begin{equation} \nonumber
\widetilde{W}(z,q) = \frac{e^{\frac{z}{2}+q}-1-\frac{1}{2}z}{z(e^{z}-1)}.
\end{equation}
\end{lemma}
\begin{proof}
    The proof is similar to that of Lemma \ref{lem:fluxes_IBP_HF_asymp}.
\end{proof}
We note that $\widetilde{W}$ is a modification of $W$ in \eqref{eq:W}, which now takes into account the effect of $Q_{j+1/2}$, and consequently $\hV$. Figure \ref{fig:W_tilde}, obtained by fixing $q=-1/2,0,1/2$, shows the typical behaviour of the function $\widetilde{W}$ for negative and positive values of $q$, and also for $q=0$. One important observation is that if $q \neq 0$, the function $\widetilde{W}$ is ill-defined when $z$ is close to zero. In terms of the differential equation, this corresponds to the case when the P\'eclet number is close to zero. We note, however, that when we are not in the advection-dominated regime, i.e., when the P\'eclet number is such that $|\mathrm{Pe}_{j+1/2}|<10$, then we can take $\alpha_j=0$ instead of \eqref{eq:alpha_val}. In this scenario, we use $\widetilde{W}(z,0)=W(z)$, which is defined near $z=0$, with $\lim_{z\rightarrow 0} W(z) = 1/8$. 
\begin{figure}
	\caption{Plots of $\widetilde{W}(z,q)$ for different values of $q$.}\label{fig:W_tilde}
 \centering
	\includegraphics[width=0.75\linewidth]{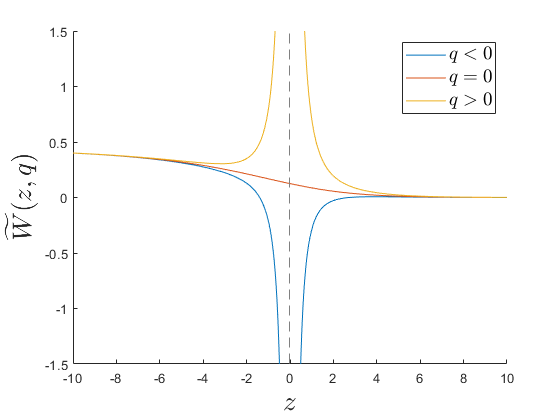} 
\end{figure}

\begin{remark}[Choice of modified P\'eclet number and fluxes]\label{rem:Peclet_mod}
	We note that both the discrete homogeneous and inhomogeneous fluxes \eqref{eq:HF_IBP_asymp} and \eqref{eq:IF_IBP_asymp}, respectively, contain two expressions. The main difference between the superscripts $+$ and $-$ comes from whether we express $\Lambda$ as a polynomial centred at $\sCoord=0$ or $\sCoord=1$. Upon choosing to express $\Lambda$ as a polynomial centred at $\sCoord=0$, we obtain homogeneous and inhomogeneous fluxes which depend on the P\'eclet number $\mathrm{Pe}_{j+1/2}^+$. On the other hand, by writing $\Lambda$ as a polynomial centred at $\sCoord=1$, the homogeneous and inhomogeneous fluxes depend on $\mathrm{Pe}_{j+1/2}^-$. Physically, taking the velocity from the upwind direction (as opposed to the downwind direction) is the correct approach, and thus the fluxes \eqref{eq:HF_IBP_pos_asymp} and \eqref{eq:IF_IBP_pos_asymp} are the appropriate choices when $V_{j+1/2}>0$, \eqref{eq:HF_IBP_neg_asymp} and \eqref{eq:IF_IBP_neg_asymp} otherwise. 
\end{remark}

\subsection{Application: coupling between the advection-diffusion equation and the Poisson equation for the velocity field}\label{sec:CF_coupled}
In this section, we relate the discussions in Sections \ref{sec:pwc}--\ref{sec:pwl} to the velocity field $V=-\frac{\rd \phy}{\rd x}$ reconstructed from the solution of the Poisson equation \eqref{eq:Poisson_1D}. Firstly, we note that the velocity field $V$ is an unknown quantity which needs to be computed from the Poisson equation \eqref{eq:Poisson_1D}, and hence applying quadrature rules to compute the integrals in \eqref{eq:HF_exact}--\eqref{eq:IF_exact} is no longer straightforward. In this section, we detail how the homogeneous and inhomogeneous fluxes in \eqref{eq:HF_exact}--\eqref{eq:IF_exact} are computed from discrete velocity fields. Since the complete flux scheme is a uniformly second-order method, a natural choice for solving the Poisson equation would be a second-order numerical scheme, e.g., the central difference method. This gives us the values $\phy_j$ of $\phy$ at the nodes $x_j, j=2,\dots, N$. We now relate the discrete velocity fields reconstructed from $\phy_j$ to the discrete fluxes \eqref{eq:HF_std}--\eqref{eq:IF_std} in Section \ref{sec:pwc}, and the modified homogeneous and inhomogeneous fluxes \eqref{eq:HF_IBP_asymp} and \eqref{eq:IF_IBP_asymp}, respectively, in Section \ref{sec:pwl}. 
\subsubsection{Piecewise constant velocity fields}\label{sec:pwc_coupled}
We start by relating the discrete velocity field to the fluxes \eqref{eq:HF_std}--\eqref{eq:IF_std}. One way to compute the velocity field $V$ at the interface $x_{j+1/2}$, is to use a second-order central difference and write
\begin{equation}\label{eq:V_ctr}
V_{j+1/2} = -\frac{\phy_{j+1}-\phy_j}{\Delta x},
\end{equation}
which upon substitution into \eqref{eq:flux_exact} gives us 
\begin{equation}\label{eq:flux_pwc}
\begin{aligned}
f_{j+1/2}
&=-\frac{D}{\Delta x} \frac{c_{j+1}e^{\frac{\mu}{D}\phy_\eh(1)}-c_{j}e^{\frac{\mu}{D}\phy_\eh(0)}}{\int_{0}^{1}e^{\frac{\mu}{D}\phy_\eh(\sCoord)} \mathrm{d}\sCoord}-\frac{\int_{0}^{1}e^{\frac{\mu}{D}\phy_\eh(\sCoord)}\hat{s}(\sCoord) \mathrm{d}\sCoord}{\int_{0}^{1}e^{\frac{\mu}{D}\phy_\eh(\sCoord)} \mathrm{d}\sCoord},
\end{aligned}
\end{equation}
where 
\begin{equation}\label{eq:sol_h}
\phy_\eh(\sCoord) = (\phy_{j+1}-\phy_j)\sCoord+\phy_j
\end{equation} is the solution to the homogeneous equation associated to the boundary value problem (in scaled coordinates) 
\begin{equation}\label{eq:BVPpoiss}
\begin{aligned}
&-\phy'' = \sP \Delta x^2 \quad \mathrm{on} \quad (0,1),\\
&\phy(0) = \phy_j, \quad \phy(1)=\phy_{j+1}.
\end{aligned}
\end{equation} 
Evaluating the expressions in \eqref{eq:flux_pwc} gives us the fluxes \eqref{eq:HF_std}--\eqref{eq:IF_std} with corresponding P\'eclet number 
\[\mathrm{Pe}_{j+1/2} = -\frac{\mu}{D}(\phy_{j+1}-\phy_{j}).
\] 
One interesting observation that can be made for the fluxes that arise from piecewise constant velocity fields is the fact that they only depend on the homogeneous solution $\phy_\eh$ of \eqref{eq:BVPpoiss}, which loses information about the particular solution $\phy_\mathrm{p}$, and consequently about the source term $\sP$ of the Poisson equation.
 
 \subsubsection{Piecewise linear velocity fields}\label{sec:pwl_coupled}
We now consider the case for piecewise linear velocity fields. As with the previous section, the constant component $V_{j+1/2}$ is computed via \eqref{eq:V_ctr}. Now, using $V=-\frac{\rd\phy}{\rd x}$ and the Poisson equation \eqref{eq:Poisson_1D}, we have that $\frac{\rd V}{\rd x}=\sP$. Classical finite volume type schemes approximate the source by piecewise constants on each cell $K_j$, with
\begin{equation}\label{eq:source_Poisson}
\sP = s_{\mathrm{P},j} \quad x\in K_j.
\end{equation}
Integrating the Poisson equation \eqref{eq:Poisson_1D} over a control volume $K_j$, and taking note that the flux is $V=-\frac{\rd \phy}{\rd x}$, we then see that
over each $K_j$, we have
\begin{equation}\label{eq:Poisson_FV}
V_{j+1/2}-V_{j-1/2} = s_{\mathrm{P},j} \Delta x .
\end{equation} 
In order to obtain an approximation for the derivative $\hV$, we consider two neighboring cells $K_j$ and $K_{j+1}$ which share the interface $x_{j+1/2}$. Adding the corresponding finite volume formulations \eqref{eq:Poisson_FV} then gives us 
\begin{equation}\nonumber
\begin{aligned}
V_{j+3/2}-V_{j-1/2} &= (s_{\mathrm{P},j}+s_{\mathrm{P},j+1}) \Delta x,
\end{aligned}
\end{equation}
or equivalently
\begin{equation}\label{eq:V'_j.5}
\frac{V_{j+3/2}-V_{j-1/2}}{2 \Delta x} = \frac{1}{2}(s_{\mathrm{P},j}+s_{\mathrm{P},j+1}):=s_{\mathrm{P},j+1/2}.
\end{equation}
We see on the left hand side of \eqref{eq:V'_j.5} a second-order central difference approximation of $\hV$, which, due to \eqref{eq:Poisson_FV}, is equivalent to $s_{\mathrm{P},j+1/2}$. Hence, the value of $\hV$ can simply be computed from the source term of the Poisson equation. We then use the expressions \eqref{eq:V_ctr} and \eqref{eq:V'_j.5} for $V_{j+1/2}$ and $\hV$ in the piecewise linear approximation \eqref{eq:V_pcwise_lin} for $V$, which upon substitution into \eqref{eq:flux_exact} gives us 
 \begin{equation}\label{eq:flux_pwl}
 \begin{aligned}
 f_{j+1/2}
 &=-\frac{D}{\Delta x}\frac{c_{j+1}e^{\frac{\mu}{D}\phy_\mathrm{C}(1)}-c_{j}e^{\frac{\mu}{D}\phy_\mathrm{C}(0)}}{\int_{0}^{1}e^{\frac{\mu}{D}\phy_\mathrm{C}(\sCoord)} \mathrm{d}\sCoord}-\frac{\int_{0}^{1}e^{\frac{\mu}{D}\phy_\mathrm{C}(\sCoord)}\hat{s}(\sCoord) \mathrm{d}\sCoord}{\int_{0}^{1}e^{\frac{\mu}{D}\phy_\mathrm{C}(\sCoord)} \mathrm{d}\sCoord},
 \end{aligned}
 \end{equation}
 where under the assumption that $\sP=s_{\mathrm{P},j+1/2}$ on $(x_j,x_{j+1})$, we have that $\phy_\mathrm{C} = \phy_\eh +\phy_\mathrm{p}$ is the complete solution to the boundary value problem \eqref{eq:BVPpoiss}, where $\phy_\mathrm{p}$ is given by
 \begin{equation}\nonumber
 \phy_\mathrm{p}(x) = s_{\mathrm{P},j+1/2}\frac{\Delta x^2}{2}(\sCoord-\sCoord^2).
 \end{equation}
 It is important to note that \eqref{eq:flux_pwl} offers a significant improvement over \eqref{eq:flux_pwc} in the sense that aside from the homogeneous solution $\phy_\eh$ to the boundary value problem \eqref{eq:BVPpoiss}, the discrete fluxes now also depend on the particular solution $\phy_{\mathrm{p}}$ to \eqref{eq:BVPpoiss}, and consequently, the source term $\sP$ of the Poisson equation. Depending on whether $\phy_\mathrm{p}$ is expressed as a polynomial centred at $\sCoord=0$ or $\sCoord=1$,  this leads to the fluxes \eqref{eq:HF_IBP_pos_asymp} and \eqref{eq:IF_IBP_pos_asymp}, or \eqref{eq:HF_IBP_neg_asymp} and \eqref{eq:IF_IBP_neg_asymp}, respectively. Here, the correction term $Q_{j+1/2}$ from \eqref{eq:Q} can be rewritten as 
 \[Q_{j+1/2}= \frac{\mu}{D} s_{\mathrm{P},j+1/2}\frac{\Delta x^2}{2}.\] 
  \section{Complete flux scheme in two dimensions} \label{sec:CF_2D}
  In this section, we discuss the complete flux scheme in two dimensions. For the discretisation, we denote by $\mesh$ a partition of the domain $\Omega$ into control volumes $K$ such that $\Omega = \bigcup_{K\in \mesh} K$. For each control volume $K$, we then denote by $\edges_K$ the collection of edges of $K$. In order to write the advection-diffusion equation \eqref{eq:adv-diff} in its finite volume form, we then take the integral over each control volume $K\in\mesh$ so that
  \[
  \int_K \nabla \bigcdot (\mu c \V - D \nabla c) \,\mathrm{d}A = \int_K s \,\mathrm{d}A.
  \]
  Defining the flux
  \[
  \mathbf{\Gamma}_c:= \mu c \V - D \nabla c
  \]
  and using the divergence theorem, we can then rewrite the equation above as
  \[
\sum_{\sigma\in\edges_K} \int_\sigma \mathbf{\Gamma}_c\cdot \mathbf{n}_{K,\sigma} \,\mathrm{d}s = \int_K s \,\mathrm{d}A,
\] 
where $\mathbf{n}_{K,\sigma}$ is the outward unit normal vector of $K$ at $\sigma$. In this paper, we consider rectangular meshes, and hence the outward unit normal vectors on the northern, southern, eastern and western edges are $\mathbf{e}_y, -\mathbf{e}_y, \mathbf{e}_x, -\mathbf{e}_x$ respectively, where $\mathbf{e}_x, \mathbf{e}_y$ are the standard basis vectors in $\mathbb{R}^2$. We also adopt the compass notation and denote by the subscripts $N,E,S,W$ the northern, eastern, southern, and western edges of cell $K$. We now discuss how to compute the flux along the eastern edge $\sigma_{E}$ of a control volume $K$. The fluxes along the other edges can be computed in a similar manner. 

Supposing that the control volumes $K, E$ have centres $(x_j,y_k)$ and $(x_{j+1},y_k)$, respectively, and a shared edge $\sigma_{E}$ described by the line segment $x=x_{j+1/2}, y\in(y_{k-1/2},y_{k+1/2})$ (see Figure \ref{fig:CF_2D}), we are required to compute
\[
\int_{\sigma_{E}} \flux_c \bigcdot \mathbf{e}_x \,\mathrm{d}s = \int_{y_{k-1/2}}^{y_{k+1/2}} \flux_c(x_{j+1/2},y) \bigcdot \mathbf{e}_x \,\mathrm{d}y.
\]
\begin{figure}
	\caption{Illustration of 2D Cartesian cell involved for computing fluxes.}\label{fig:CF_2D}
	\centering
	\includegraphics[width=0.45\linewidth]{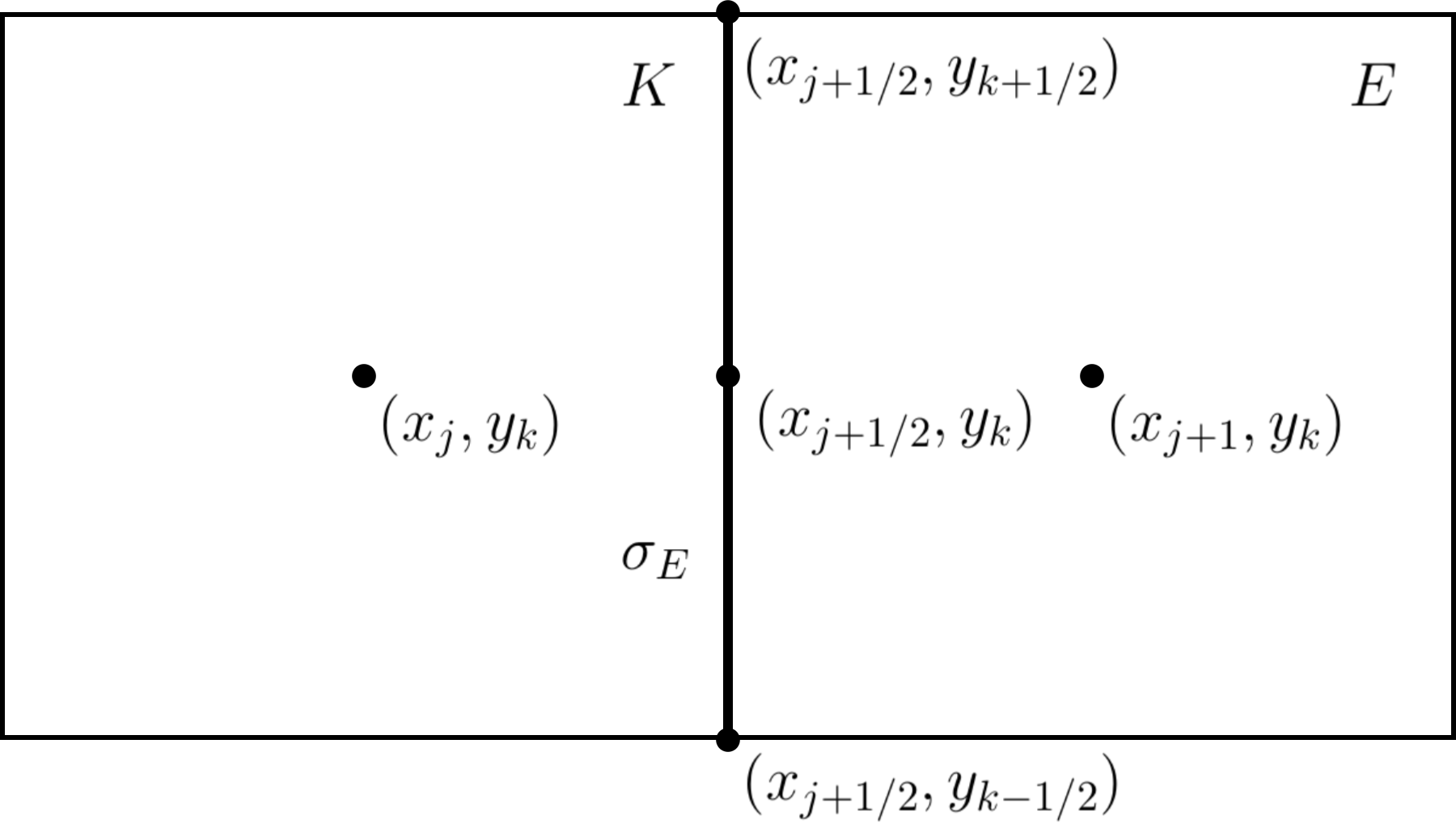}
\end{figure}
We use the midpoint rule and approximate the integral of the flux by
\[
\int_{y_{k-1/2}}^{y_{k+1/2}} \flux_c(x_{j+1/2},y) \bigcdot \mathbf{e}_x \,\mathrm{d}y \approx \Delta y \flux_c(x_{j+1/2},y_k) \bigcdot \mathbf{e}_x =: F_{K,\sigma_{E}}.
\]
We are then left to determine the value of the flux at $(x_{j+1/2},y_k)$. As with the 1D case, the flux will be computed via an associated boundary value problem on $y=y_k$, $x \in (x_j,x_{j+1})$, which is given by 
\begin{equation}\nonumber
\begin{aligned}
\frac{\partial }{\partial x}\big(\flux_c \bigcdot \mathbf{e}_x \big) =& s -  \frac{\partial }{\partial y}\big(\flux_c \bigcdot \mathbf{e}_y \big) =: \tilde{s}, \quad x\in(x_j,x_{j+1}), \,\,\,y=y_k,\\
c(x_{j},y_k) = c_{j,k}, &\quad  c(x_{j+1},y_k) = c_{j+1,k}. 
\end{aligned}
\end{equation}
By treating the partial differential equation as a quasi-one-dimensional ODE, i.e., treating $y$ as a constant, the expression for $\flux_c \cdot \mathbf{e}_x $ can then be obtained by the same process as the one described in Section \ref{sec:CF_1D}. In particular, for piecewise constant velocity fields, we still have 
\begin{subequations} \label{eq:fluxes_std_2D}
\begin{align} 
F_{K,\sigma_E}^\eh &= -\frac{D\Delta y}{\Delta x} \big( B(\mathrm{Pe}_{j+1/2,k})c_{j+1,k}-B(-\mathrm{Pe}_{j+1/2,k})c_{j,k}\big), \label{eq:HF_std_2D}\\
F_{K,\sigma_E}^\ei &= -\Delta y \Delta x (W(\mathrm{Pe}_{j+1/2,k})\tilde{s}_{j+1,k}^E- W(-\mathrm{Pe}_{j+1/2,k})\tilde{s}_{j,k}^E) \label{eq:IF_std_2D},
\end{align} 
\end{subequations}
where $F_{K,\sigma_E}^\eh, F_{K,\sigma_E}^\ei$ are the homogeneous and inhomogeneous components of the integrated flux, respectively. Compared to the complete flux scheme in one dimension \eqref{eq:HF_std}--\eqref{eq:IF_std}, there are only two main changes. Firstly, since the velocity field $\Vel$ is now a vector, the definition of the P\'eclet number \eqref{eq:Pec_std} needs to be modified so that we have 
\begin{equation}\label{eq:Pec_std_2D}
\mathrm{Pe}_{j+1/2,k} = \frac{\mu}{D} \Vel_{j+1/2,k} \bigcdot \mathbf{e}_x \Delta x.
\end{equation} We note that even though the numerical scheme was presented under the assumption that the diffusion parameter $D$ is a scalar, an extension to the case where $D$ is an anisotropic diffusion tensor is possible by following the ideas in \cite{CT20-complete_flux}. Secondly, aside from the source term $s$, we now also have the term  $\frac{\partial}{\partial y}\big(\flux_c \bigcdot \mathbf{e}_y \big)$, which, following the nomenclature in \cite{AB11-FVCF}, we refer to as a \emph{cross flux}, on the right hand side of the differential equation. Hence, $\tilde{s}_{j,k}^E, \tilde{s}_{j+1,k}^E$ refer to piecewise constant approximations of the total source $\tilde{s}= s - \frac{\partial }{\partial y}\big(\flux_c \cdot \mathbf{e}_y \big)$ over the regions $(x_{j},x_{j+1/2})\times(y_{k-1/2},y_{k+1/2})$ and $(x_{j+1/2},x_{j+1})\times(y_{k-1/2},y_{k+1/2})$, respectively.  For $s$, we may simply take the average value of the source terms $s_{j,k}$ and $s_{j+1,k}$ in cells $K$ and $E$, respectively. The main challenge is then to find an approximation for the cross flux $\frac{\partial}{\partial y}\big(\flux_c \bigcdot \mathbf{e}_y \big)$. For $x\in(x_j,x_{j+1/2})$, we use the average value of the cross flux over the region $(x_{j},x_{j+1/2})\times(y_{k-1/2},y_{k+1/2})$, given by
\begin{equation}\nonumber
\begin{aligned} 
\frac{2}{\Delta x \Delta y}
&\int_{x_j}^{x_{j+1/2}}  \int_{y_{k-1/2}}^{y_{k+1/2}}\frac{\partial}{\partial y}\bigg(\mathbf{\Gamma}_c \bigcdot \mathbf{e}_y \bigg) \mathrm{d}y\mathrm{d}x \\ 
=& \frac{2}{\Delta x \Delta y} \bigg(\int_{x_j}^{x_{j+1/2}} \mathbf{\Gamma}_c(x,y_{k+1/2})\bigcdot \mathbf{e}_y\, \mathrm{d}x - \int_{x_j}^{x_{j+1/2}} \mathbf{\Gamma}_c(x,y_{k-1/2})\bigcdot \mathbf{e}_y \, \mathrm{d}x \bigg).
\end{aligned} 
\end{equation}
Here, we notice that the terms on the right hand side are integrated fluxes along the northern and southern edges of $K$, respectively. Since the approximation for $s$ is piecewise constant, it is sufficient to use a first-order approximation for the average value of the flux $\mathbf{\Gamma}_c \bigcdot \mathbf{e}_y$. This can be done by taking the homogeneous component of the fluxes along the northern and southern edges, which yields the following discrete approximation for the average value of the cross flux on $(x_{j},x_{j+1/2})$
\begin{figure}
	\caption{Fluxes along the edges of a cell $K$.}\label{fig:fluxes_2D}
	\centering
	\includegraphics[width=0.45\linewidth]{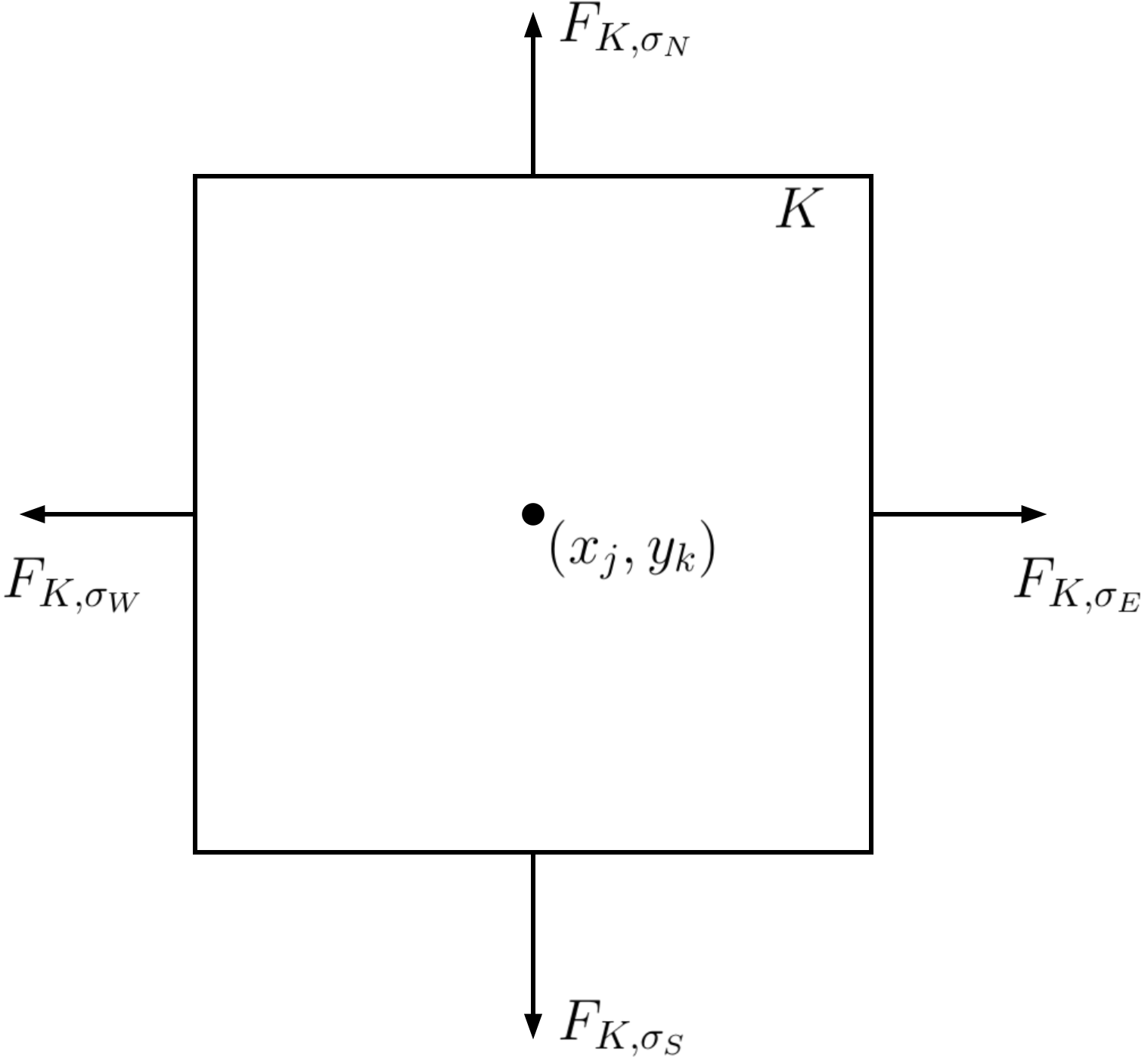}
\end{figure}
\[
\frac{\partial}{\partial y} \big(\Gamma_c\bigcdot \mathbf{e}_y\big) \approx \frac{1}{\Delta x \Delta y} (F_{K,\sigma_{N}}^\eh+F_{K,\sigma_{S}}^\eh),
\]
where $F_{K,\sigma_N}^\eh$ and $F_{K,\sigma_S}^\eh$ denote the homogeneous component of the integrated fluxes along the northern and southern edges of cell $K$, respectively (see Figure \ref{fig:fluxes_2D}). By a similar approach, the average value of the cross flux at $(x_{j+1/2},x_{j+1})$ is approximated by 
\[
\frac{1}{\Delta x \Delta y} (F_{E,\sigma_{N}}^\eh+F_{E,\sigma_{S}}^\eh),
\]
where $F_{E,\sigma_N}^\eh$, $F_{E,\sigma_S}^\eh$ are the homogeneous components of the integrated fluxes along the northern and southern edges of cell $E$, respectively. To summarise, we have 
\begin{subequations}
\begin{align}
\tilde{s}_{j,k}^E &= s_{j,k} - \frac{1}{\Delta x \Delta y} (F_{K,\sigma_{N}}^\eh+F_{K,\sigma_{S}}^\eh),\label{eq:totsrc_j} \\
\tilde{s}_{j+1,k}^E &= s_{j+1,k} - \frac{1}{\Delta x \Delta y} (F_{E,\sigma_{N}}^\eh+F_{E,\sigma_{S}}^\eh).\label{eq:totsrc_j1}
\end{align}
\end{subequations}
We now aim to  extend the complete flux scheme for piecewise linear velocity fields in Section \ref{sec:pwl} to 2D. In order to do so, we require an expression for the correction term $Q_{j+1/2}$ in 2D; cf. \eqref{eq:Q}. By denoting $\Vel = (V_1, V_2)^T$, we can write
\begin{equation} \label{eq:Q_2D}
Q_{j+1/2,k} = \frac{\mu}{D}\frac{\partial}{\partial x} V_1(x_{j+1/2},y_k) \frac{\Delta x^2}{2}.
\end{equation}
In two dimensions, the upwind-adjusted P\'eclet numbers \eqref{eq:Peclet_pos}--\eqref{eq:Peclet_neg} can then be written as 
\begin{equation}\nonumber
\begin{aligned}
\mathrm{Pe}^+_{j+1/2,k} &= \mathrm{Pe}_{j+1/2,k} -\alpha_{j,k}Q_{j+1/2,k}, \\
\mathrm{Pe}^-_{j+1/2,k} &= \mathrm{Pe}_{j+1/2,k} + \alpha_{j,k} Q_{j+1/2,k},
\end{aligned}
\end{equation}
$\alpha_{j,k}\in[0,1]$. Here, $\alpha_{j,k}$ is chosen in a similar manner as in \eqref{eq:alpha_val} so that there is no over-correction. Having defined the correction term and the upwind-adjusted P\'eclet numbers, we can now write the integrated discrete homogeneous fluxes as
\begin{subequations} \label{eq:fluxes_2D_HF}
\begin{align} 
F_{K,\sigma_E}^{\eh,+} &= -\frac{D\Delta y}{\Delta x} \big( e^{- \alpha_{j,k}Q_{j+1/2,k}}B(\mathrm{Pe}^+_{j+1/2,k})c_{j+1,k}-B(-\mathrm{Pe}^+_{j+1/2,k})c_{j,k}\big), \label{eq:HF_pwl_2D_pos}\\
F_{K,\sigma_E}^{\eh,-} &= -\frac{D\Delta y}{\Delta x} \big( B(\mathrm{Pe}^-_{j+1/2,k})c_{j+1,k}-e^{ -\alpha_{j,k}Q_{j+1/2,k}}B(-\mathrm{Pe}^-_{j+1/2,k})c_{j,k}\big), \label{eq:HF_pwl_2D_neg}
\end{align} 
\end{subequations}
which correspond to \eqref{eq:HF_IBP_asymp}. Similarly, the integrated discrete inhomogeneous fluxes corresponding to \eqref{eq:IF_IBP_asymp} are given by
\begin{subequations} \label{eq:fluxes_2D_IF}
	\begin{align} 
 F_{K,\sigma_E}^{\ei,+} &= -\Delta y \Delta x \bigg(\widetilde{W}(\mathrm{Pe}^+_{j+1/2,k},-\frac{3}{4}\alpha_{j,k}Q_{j+1/2,k})\tilde{s}_{j+1,k}^E \nonumber\\
&\,\,\,\qquad\qquad\qquad- \widetilde{W}(-\mathrm{Pe}^+_{j+1/2,k},\frac{1}{4}\alpha_{j,k}Q_{j+1/2,k})\tilde{s}_{j,k}^E\bigg), \label{eq:IF_pwl_2D_pos}
	\\
	F_{K,\sigma_E}^{\ei,-} &= -\Delta y \Delta x \bigg(\widetilde{W}(\mathrm{Pe}^-_{j+1/2,k},-\frac{1}{4}\alpha_{j,k}Q_{j+1/2,k})\tilde{s}_{j+1,k}^E\nonumber\\
	&\,\,\,\qquad\qquad\qquad-  \widetilde{W}(-\mathrm{Pe}^-_{j+1/2,k},-\frac{5}{4}\alpha_{j,k}Q_{j+1/2,k})\tilde{s}_{j,k}^E\bigg). \label{eq:IF_pwl_2D_neg}
	\end{align} 
\end{subequations} 

We note that for the inhomogeneous component of the integrated fluxes \eqref{eq:IF_pwl_2D_pos} and  \eqref{eq:IF_pwl_2D_neg}, the total source \eqref{eq:totsrc_j}--\eqref{eq:totsrc_j1} uses the corresponding homogeneous components along the northern and southern edges which are similar to \eqref{eq:HF_pwl_2D_pos} and \eqref{eq:HF_pwl_2D_neg}, respectively. As with the discussion in Section \ref{sec:CF_coupled}, we see that the upwind homogeneous and inhomogeneous fluxes \eqref{eq:fluxes_2D_HF} and \eqref{eq:fluxes_2D_IF}, respectively,  now take into account the first-order derivatives of the velocity field $\V$, and are thus natural improvements over the fluxes \eqref{eq:HF_std_2D}--\eqref{eq:IF_std_2D}. Moreover, following Remark \ref{rem:Peclet_mod}, we see that the fluxes \eqref{eq:HF_pwl_2D_pos} and \eqref{eq:IF_pwl_2D_pos} are the proper choices whenever $\mathrm{Pe}_{j+1/2,k}>0$, \eqref{eq:HF_pwl_2D_neg} and \eqref{eq:IF_pwl_2D_neg} otherwise.
 
\section{Numerical tests} \label{sec:NumTests}
In this section, we perform numerical tests for the advection-diffusion equation \eqref{eq:adv-diff}, coupled to the Poisson equation \eqref{eq:Poisson} for the velocity $\Vel$. For these tests, we prescribe a solution $c^*$ and calculate the source term $s$ for equation \eqref{eq:adv-diff} accordingly. Dirichlet boundary conditions are then imposed. For the Poisson equation, a second-order central difference scheme is used to numerically calculate $\phy$, which is then used to construct $\Vel$. For a given solution $c^*$ of the advection-diffusion equation \eqref{eq:adv-diff}, we measure the relative error in the $L^2$-norm by computing
\begin{equation}\nonumber
\norm{E}{2} := \frac{\norm{c^*-c}{2}}{\norm{c^*}{2}},
\end{equation}
where $c$ is the piecewise constant function reconstructed from the discrete solution. That is, $c=c_K$ for each $x \in K, K\in\mesh$. For all of the tests considered below, choosing $\alpha_j=1$ for all $j$ satisfies the requirement \eqref{eq:alpha_val}.

\subsection{1D tests}
We start by performing numerical tests in one dimension on the domain $\Omega=(0,1)$ with $N$ equidistant cells, i.e., $\Delta x = \frac{1}{N}$. Here, the velocity $V$ and its derivative $V_x$ are reconstructed as in \eqref{eq:V_ctr} and \eqref{eq:V'_j.5}, respectively.  
\subsubsection{Test case 1}
We start by considering a numerical test from \cite[Section 8.1]{AB11-FVCF}, where we take $V=1-0.95\sin(\pi x)$, and 
\begin{equation}\nonumber
c^*(x) = 0.2\sin(\pi x) + \frac{e^{(x-1)/D}-e^{-1/D}}{1-e^{-1/D}}.
\end{equation}
For this test case, $V>0$ and hence \eqref{eq:HF_IBP_pos_asymp}  and \eqref{eq:IF_IBP_pos_asymp} are upwind fluxes whereas \eqref{eq:HF_IBP_neg_asymp} and \eqref{eq:IF_IBP_neg_asymp} are downwind fluxes. In the numerical results presented below, $V$ and $V_x$ are reconstructed from the discrete solution of the Poisson equation $-\frac{\rd^2\phy}{\rd x^2} = -0.95\pi\cos(\pi x)$ with appropriate Dirichlet boundary conditions. We consider first the case $D=1$, corresponding to a test with dominant diffusion.  
\begin{table}[h!]
	\caption{Relative errors in the solution profile, test case 1, $D=1$, complete flux scheme.}\label{tab:test1.diff_dom}
	\centering
	\begin{tabular}{c|c c|c c| c c}
		\hline
		$N$ &  \multicolumn{2}{|c|}{flux \eqref{eq:fluxes_std} }  & \multicolumn{2}{|c|}{ upwind \eqref{eq:HF_IBP_pos_asymp} \& \eqref{eq:IF_IBP_pos_asymp}} & \multicolumn{2}{|c}{ downwind \eqref{eq:HF_IBP_neg_asymp} \& \eqref{eq:IF_IBP_neg_asymp} }   \\
		& $\norm{E}{2}$ & order & 
		$\norm{E}{2}
		$ & order & $\norm{E}{2}$ & order\\
		\hline
		40 & 1.3957e-4 & - & 2.5960e-5 & - & 2.5739e-4 & - \\
		80 & 3.5942e-5 & 1.9573 & 6.5651e-6 & 1.9834 & 6.6507e-5 & 1.9524 \\
		160 & 9.1268e-6 & 1.9775 & 1.6536e-6 & 1.9892 &1.6913e-5 &1.9754 \\
		320 & 2.3001e-6 & 1.9884 & 4.1518e-7 & 1.9938 & 4.2650e-6 & 1.9875 \\
		640 & 5.7738e-7 & 1.9941 & 1.0404e-7 & 1.9966 & 1.0709e-6 & 1.9937 \\
		1280 & 1.4464e-7 & 1.9970 & 2.6038e-8 & 1.9984 & 2.6883e-7 & 1.9968 \\
		\hline
	\end{tabular} 
\end{table}

In Table \ref{tab:test1.diff_dom}, we see that in the diffusion dominated regime, the numerical results obtained from all of the complete fluxes: piecewise constant velocity \eqref{eq:fluxes_std},  upwind \eqref{eq:HF_IBP_pos_asymp} and \eqref{eq:IF_IBP_pos_asymp}, and downwind \eqref{eq:HF_IBP_neg_asymp} and \eqref{eq:IF_IBP_neg_asymp} exhibit second-order accuracy. We note, however, that the downwind fluxes yielded slightly less accurate results. We now consider an advection dominated test case by taking $D=10^{-8}$. In particular, this test is interesting because for $D\ll1$, $c^*$ has a very thin boundary layer near $x=1$. 

\begin{table}[h]
	\caption{Relative errors in the solution profile, test case 1, $D=10^{-8}$, complete flux scheme.}\label{tab:test1.adv_dom}
	\centering
	\begin{tabular}{c|c c|c c| c c}
		\hline
		$N$ &  \multicolumn{2}{|c|}{flux \eqref{eq:fluxes_std} } & \multicolumn{2}{|c|}{ upwind \eqref{eq:HF_IBP_pos_asymp} \& \eqref{eq:IF_IBP_pos_asymp}} & \multicolumn{2}{|c}{ downwind \eqref{eq:HF_IBP_neg_asymp} \& \eqref{eq:IF_IBP_neg_asymp} } \\
		& $\norm{E}{2}$ & order & 
		$\norm{E}{2}
		$ & order & $\norm{E}{2}$ & order\\
		\hline
		40 & 6.5971e-2 & - & 2.5940e-2 & - & 1.2146e-1 & -\\
		80 & 3.6815e-2 & 0.8415 & 7.6406e-3 & 1.7634 & 7.2024e-2 & 0.7539 \\
		160 & 2.0240e-2 & 0.8630 & 2.1286e-3 & 1.8437 &4.0256e-2 &0.8393 \\
		320 & 1.0765e-2 & 0.9109 & 5.6762e-4 & 1.9070 & 2.1500e-2 & 0.9048\\
		640 & 5.5758e-3 & 0.9491 & 1.4706e-4 & 1.9485 & 1.1484e-2 &   0.9475 \\
		1280 & 2.8408e-3 & 0.9729 & 3.7470e-5 & 1.9726 & 5.6819e-3 &   0.9724 \\
		\hline
	\end{tabular} 
\end{table}

Table \ref{tab:test1.adv_dom} shows us that in the advection-dominated regime, the numerical solution obtained from a piecewise constant approximation of the velocity, which corresponds to the numerical fluxes \eqref{eq:fluxes_std}, is now only first-order accurate. We also note that the numerical solution obtained from the downwind fluxes  \eqref{eq:HF_IBP_neg_asymp} and \eqref{eq:IF_IBP_neg_asymp}, is also first-order accurate. On the other hand, the numerical solution obtained from the upwind fluxes  \eqref{eq:HF_IBP_pos_asymp} and \eqref{eq:IF_IBP_pos_asymp} results in a second-order accurate approximation of the solution. This agrees with the observation made in Remark \ref{rem:Peclet_mod} that the proper choice for the modified P\'eclet number and fluxes should be taken from the upwind direction.

\subsubsection{Test case 2}
For our second test case, we consider an advection-dominated problem by taking $D=10^{-8}$, and suppose that the exact solution for the advection-diffusion equation is given by 
\begin{equation}\nonumber
c^*(x)=\sin(\pi x).
\end{equation} Moreover, we work with an unknown velocity $V$ such that the source term of the Poisson equation is given by
\[
\sP = -A (e^{-1000x^2}-e^{-1000(1-x)^2}).
\] We then impose Dirichlet boundary conditions $\phy(1)=0, \phy(0)=-300$. Here, we observe that $\left \lvert \frac{\rd V}{\rd x} \right \rvert \ll 1 $ in most of the domain, except near the boundaries where $\left \lvert \frac{\rd V}{\rd x} \right \rvert \gg 1 $. This aims to mimic what is commonly encountered in plasma physics or porous media applications, for which $\frac{\rd V}{\rd x} = 0$ almost everywhere, except for a very small part of the domain; the regions near the boundary correspond to sources or sinks. For this test, $V<0$ and hence the upwind fluxes are given by \eqref{eq:HF_IBP_neg_asymp} and \eqref{eq:IF_IBP_neg_asymp}. We start by considering $A=10$, which corresponds to a moderately steep source term $\sP$. 

\begin{table}[h]
	\caption{Relative errors in the solution profile, test case 2, $A=10$, complete flux scheme.}\label{tab:Psrc_A10}
	\centering
	\begin{tabular}{c|c c|c c| c c}
		\hline
	$N$ &  \multicolumn{2}{|c|}{flux \eqref{eq:fluxes_std} }  & \multicolumn{2}{|c|}{ downwind \eqref{eq:HF_IBP_pos_asymp} \& \eqref{eq:IF_IBP_pos_asymp}} & \multicolumn{2}{|c}{ upwind \eqref{eq:HF_IBP_neg_asymp} \& \eqref{eq:IF_IBP_neg_asymp} }   \\
		& $\norm{E}{2}$ & order & 
		$\norm{E}{2}
		$ & order & $\norm{E}{2}$ & order\\
		\hline
		40 & 6.0872e-4 & - & 6.0775e-4 & - & 6.0977e-4 & - \\
		80 & 1.3097e-4 & 2.2166 & 1.3096e-4 & 2.2143 & 1.3105e-4 & 2.2181 \\
		160 & 3.0417e-5 & 2.1063 & 3.0512e-5 & 2.1017 &3.0393e-5 &2.1083 \\
		320 & 7.3620e-6 & 2.0467 & 7.4681e-6 & 2.0306 & 7.3275e-6 &  2.0524 \\
		640 & 1.8357e-6 & 2.0038 & 1.9405e-6 & 1.9443 & 1.7995e-6 & 2.0257 \\
		1280 & 4.8153e-7 & 1.9306 & 5.7534e-6 & 1.7540 & 4.4594e-7 & 2.0127 \\
		\hline
	\end{tabular} 
\end{table}

In Table \ref{tab:Psrc_A10}, we see that when $A=10$, all three numerical fluxes exhibit second-order accuracy, and the results are very close to each other. To further understand this behaviour, we look at Figure \ref{fig:V_src_A10}. Here, we see that due to the fact that the source term of the Poisson equation is not too steep, the velocity field $V \in (-300,-299.7)$ is almost constant, and hence using a piecewise constant approximation of $V$, which gives the fluxes \eqref{eq:fluxes_std}, is sufficient to obtain a second-order scheme. Moreover, due to the fact that $A=10$, the contribution of $Q_{j+1/2}$ to the modified P\'eclet number is minimal; hence resulting in numerical solutions which are very close to each other.
\begin{figure}
	\caption{Velocity and source term of Poisson equation, $A=10$ (left: $V$; right: $\sP$).}\label{fig:V_src_A10}
	\begin{tabular}{cc}
		\includegraphics[width=0.5\linewidth]{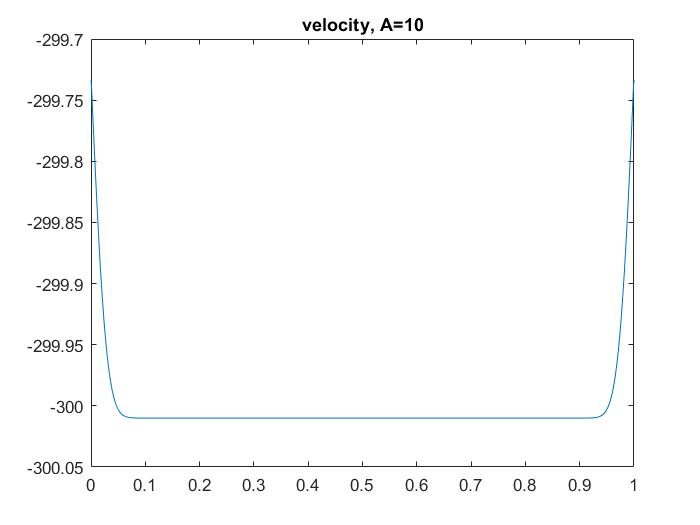} & \includegraphics[width=0.5\linewidth]{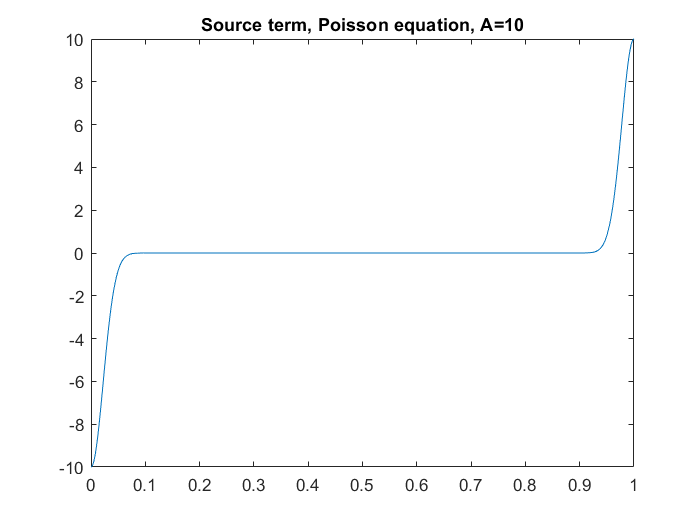} 
	\end{tabular}
\end{figure}

To further test the scheme, we consider a very steep source term, by taking $A=1000$. We now observe in Table \ref{tab:Psrc_A1000} that the numerical solutions obtained by using the fluxes \eqref{eq:fluxes_std}, and \eqref{eq:HF_IBP_pos_asymp} and \eqref{eq:IF_IBP_pos_asymp}, are only first-order accurate, whereas using the upwind fluxes \eqref{eq:HF_IBP_neg_asymp} and \eqref{eq:IF_IBP_neg_asymp} gives solutions which are second-order accurate. This is due to the fact that with a steeper source term, the velocity $V$ is no longer almost constant, as can be seen in Figure \ref{fig:V_src_A1000}. Moreover, since $A=1000$, the contribution of the correction term $Q_{j+1/2}$ is no longer negligible, and hence, making an adjustment in the incorrect direction and taking the downwind fluxes only gives a first-order accurate solution. 
\begin{table}[h]
	\caption{Relative errors in the solution profile, test case 2, $A=1000$, complete flux scheme.}\label{tab:Psrc_A1000}
		\centering
		\begin{tabular}{c|c c|c c| c c}
			\hline
		$N$ &  \multicolumn{2}{|c|}{flux \eqref{eq:fluxes_std} }  & \multicolumn{2}{|c|}{ downwind \eqref{eq:HF_IBP_pos_asymp} \& \eqref{eq:IF_IBP_pos_asymp}} & \multicolumn{2}{|c}{ upwind \eqref{eq:HF_IBP_neg_asymp} \& \eqref{eq:IF_IBP_neg_asymp} }  \\
			 & $\norm{E}{2}$ & order & 
			$\norm{E}{2}
			$ & order & $\norm{E}{2}$ & order\\
			\hline
			40 & 3.1242e-3 & - & 3.4986e-3 & - & 2.9395e-3 & - \\
			80 & 1.4130e-3 & 1.1447 & 1.5646e-3 & 1.1610 & 1.3323e-3 & 1.1417 \\
			160 & 4.2599e-4 & 1.7299 & 5.1057e-4 & 1.6157 &3.8726e-4 &1.7826 \\
			320 & 1.2728e-5 & 1.7429 & 1.8336e-4 & 1.4775 & 1.0047e-4 & 1.9465 \\
			640 & 4.5457e-5 & 1.4854 & 7.9277e-5 & 1.2097 & 2.5354e-5 & 1.9865 \\
			1280 & 1.9761e-5 & 1.2018 & 3.7918e-5 & 1.0640 & 6.3543e-6 & 1.9964 \\
			\hline
		\end{tabular} 
\end{table}

\begin{figure}
	\caption{Velocity and source term of Poisson equation, $A=1000$ (left: $V$; right: $\sP$).}\label{fig:V_src_A1000}
	\begin{tabular}{cc}
		\includegraphics[width=0.5\linewidth]{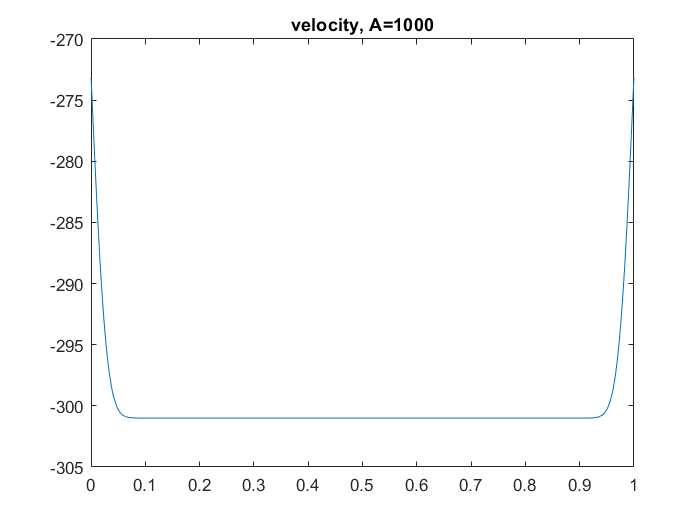} & \includegraphics[width=0.5\linewidth]{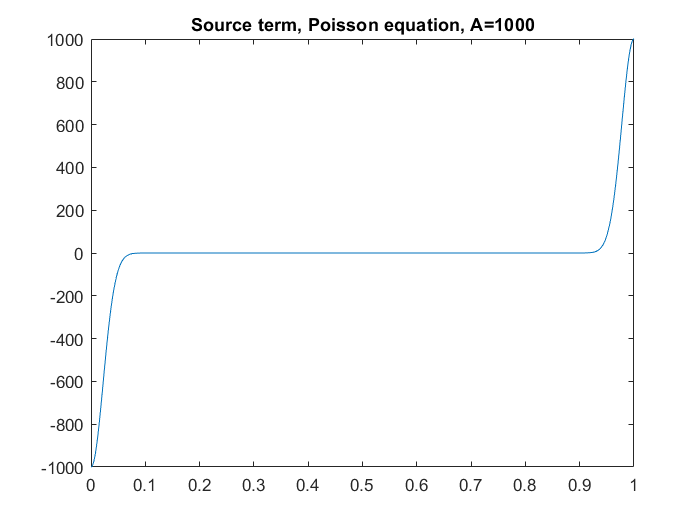} 
	\end{tabular}
\end{figure}

\subsection{2D tests} In this section, we perform 2D tests on the domain $\Omega=(0,1)\times(0,1)$ with Cartesian meshes consisting of $N\times N$ square cells. We note that since we consider Cartesian meshes, the P\'eclet number \eqref{eq:Pec_std_2D} only requires one component of the velocity field for each edge. To be specific, writing $\Vel = (V_1,V_2)^T$, we have
\[
\Vel_{j+1/2,k} \cdot \mathbf{e}_x = V_1(x_{j+1/2},y_k),
\] 
which can be approximated via central differences by
\[
V_1(x_{j+1/2},y_k) \approx -\frac{\phy_{j+1,k}-\phy_{j,k}}{\Delta x}.
\]
For the homogeneous and inhomogeneous fluxes \eqref{eq:fluxes_2D_HF}--\eqref{eq:fluxes_2D_IF}, an approach similar to \eqref{eq:V'_j.5} is then used to approximate the correction terms \eqref{eq:Q_2D}. 
\subsubsection{Test case 3}
We now consider a test case with 
\[
c^*(x,y) = \sin(\pi x) \sin(\pi y)
\]
and a velocity field $\Vel$ that is determined by the following solution to the Poisson equation 
\[
\phy(x,y) = \sin(\pi x) \sin(\pi y) + \sin(2\pi x) \sin(2\pi y)+9x+9y.
\]
This describes a flow that moves from the upper right region towards the lower left region of the domain (see Figure \ref{fig:vel_streamlines}). Since $V_1, V_2<0$ for this test case, the relevant upwind fluxes are \eqref{eq:HF_pwl_2D_neg} and \eqref{eq:IF_pwl_2D_neg}. 
\begin{figure}
	\caption{Streamlines of the velocity field $\Vel$, test case 3.}\label{fig:vel_streamlines}
	\centering
	\includegraphics[width=0.45\linewidth]{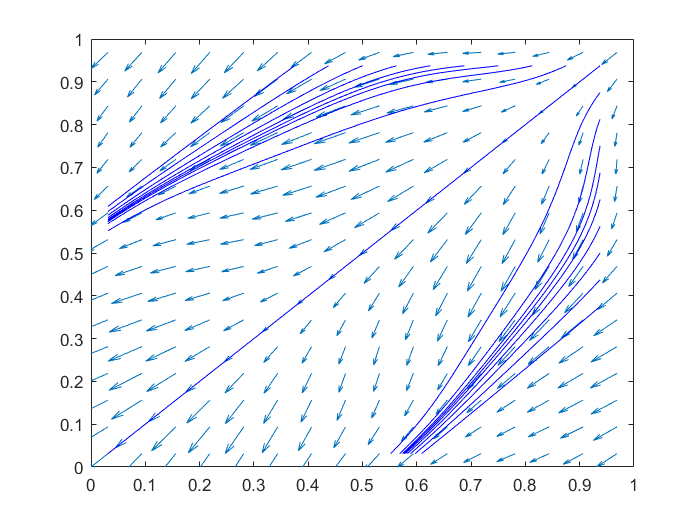}
\end{figure}

We start by considering a diffusion dominated regime by taking $D=1$. Upon looking at Table \ref{tab:test3_diff}, we see that, as in the test cases in one dimension, using the complete flux scheme with a piecewise constant approximation of the velocity field already gives a good enough approximation to the solution, with second-order accuracy. In this case, the correction term \eqref{eq:Q_2D} for the velocity field only offers a slight improvement in the accuracy of the numerical solution. On the other hand, when considering an advection-dominated regime, by taking $D=10^{-8}$, we see in Table \ref{tab:test3_adv} that using the modified fluxes in the proper upwind direction, in this case \eqref{eq:HF_pwl_2D_neg} and \eqref{eq:IF_pwl_2D_neg}, helps ensure that the complete flux scheme remains second-order accurate. We remark that for other types of velocity fields, the proper upwind fluxes might be a combination of \eqref{eq:HF_pwl_2D_pos} and \eqref{eq:IF_pwl_2D_pos}, with \eqref{eq:HF_pwl_2D_neg} and \eqref{eq:IF_pwl_2D_neg}, depending on the sign(s) of $\Vel \bigcdot \mathbf{e}_x$ and $\Vel \bigcdot \mathbf{e}_y$. 
	\begin{table}[h]
	\caption{Relative errors in the solution profile, test case 3, $D=1$.} \label{tab:test3_diff}
		\centering
\begin{tabular}{c|c c|c c}
	\hline
	$N$ &  \multicolumn{2}{|c|}{flux \eqref{eq:fluxes_std_2D} }  & \multicolumn{2}{|c}{ upwind flux \eqref{eq:HF_pwl_2D_neg} \& \eqref{eq:IF_pwl_2D_neg} }    \\
	& $\norm{E}{2}$ & order & 
	$\norm{E}{2}
	$ & order  \\
	\hline
		16$\times$16 & 1.4674e-2 & - & 1.1534e-2 & - \\
		32$\times$32 & 3.9639e-3 & 1.8883 & 3.2077e-3 & 1.8463 \\
		64$\times$64 & 1.0434e-3 & 1.9256 & 8.6021e-4 & 1.8988 \\
		128$\times$128 & 2.6859e-4 & 1.9579 & 2.2366e-4 & 1.9434 \\
		256$\times$256 & 6.8194e-5 & 1.9777 & 5.7080e-5 & 1.9702 \\
		\hline
	\end{tabular} 
\end{table}

	\begin{table}[h]
	\caption{Relative errors in the solution profile, test case 3, $D=10^{-8}$.} \label{tab:test3_adv}
	\centering
	\begin{tabular}{c|c c|c c}
		\hline
		$N$ &  \multicolumn{2}{|c|}{flux \eqref{eq:fluxes_std_2D} }  & \multicolumn{2}{|c}{ upwind flux \eqref{eq:HF_pwl_2D_neg} \& \eqref{eq:IF_pwl_2D_neg}  }    \\
		& $\norm{E}{2}$ & order & 
		$\norm{E}{2}
		$ & order  \\
		\hline
		16$\times$16 & 1.1226e-1 & - & 1.1204e-1 & - \\
	32$\times$32 & 3.9987e-2 & 1.4894 & 3.1240e-2 & 1.8426 \\
	64$\times$64 & 1.7269e-2 & 1.2113 & 8.3005e-3 & 1.9122 \\
	128$\times$128 & 8.3230e-3 & 1.0531 & 2.1406e-3 & 1.9552 \\
	256$\times$256 & 4.1431e-3 & 1.0064 & 5.4360e-4 & 1.9774 \\
		\hline
	\end{tabular} 
\end{table}
\section{Summary and future work} \label{sec:Conc}
In this paper, we considered a finite volume complete flux scheme for an advection-diffusion equation, coupled to a Poisson equation for the velocity field. By using a piecewise constant approximation of the velocity field, we obtained the classic homogeneous and inhomogeneous fluxes \eqref{eq:fluxes_std} in 1D, and \eqref{eq:fluxes_std_2D} in 2D. These fluxes resulted in a numerical scheme which is second-order accurate in the diffusion-dominated regime. However, in the advection-dominated regime, if the velocity field is nonconstant, we found that the numerical scheme reduces to first-order. This can further be explained by the fact that these discrete fluxes only take into account the homogeneous solution to the associated boundary value problem from the Poisson equation, which was insufficient. The main novelty of this paper comes from the use of a piecewise linear approximation of the velocity field, which results in homogeneous and inhomogeneous fluxes that now take into account the complete solution to the associated boundary value problem from the Poisson equation. Another novelty introduced in this paper is the upwind-adjusted P\'eclet number, which ensures that the discrete homogeneous and inhomogeneous fluxes resulting from the piecewise linear approximations of the velocity field are adjusted in the proper direction, resulting in the second-order accuracy of the complete flux scheme for nonconstant velocity fields. Although the numerical scheme and tests were performed under the assumption that the diffusion parameter $D$ is a scalar, the scheme may be extended to cover anisotropic diffusion by following the ideas in \cite{CT20-complete_flux}. Of particular interest in porous media and plasma physics applications are very steep velocity fields which satisfy $\nabla \bigcdot \V = 0$ almost everywhere, except possibly for only a very small region of the domain. Test case 2 in Section \ref{sec:NumTests} shows that the novel complete flux scheme developed in this paper is able to handle similar types of velocity fields, whilst maintaining second-order accuracy of the scheme. 

A natural extension for this work would be to fully integrate the complete flux scheme to time-dependent advection-diffusion equations, which are  encountered in mathematical models for porous media and plasma physics applications. Another avenue for future research would be to extend the complete flux scheme for equations which involve nonlinear advection, which allows for a wider class of applications, such as the Navier-Stokes equations. 
\section{Acknowledgements}
The authors would like to thank Prof. Martijn Anthonissen for his comments, which improved the overall presentation of this work. The authors would also like to thank the referees for their careful reading of the manuscript and for their feedback, which helped improve the quality of this work. 
	\bibliographystyle{abbrv}
\bibliography{CF_vel}
\end{document}